\documentclass[11pt]{amsart}
\usepackage{amssymb,latexsym}

\newcommand{\C}{\mathbb{C}}

\newcommand{\Z}{\mathbb{Z}}

\newcommand{\M}{\mathcal{M}}

\def\be{\begin{equation}}
\def\ee{\end{equation}}

\newtheorem{theorem}{Theorem}[section]
\newtheorem{prop}[theorem]{Proposition}
\newtheorem{lemma}[theorem]{Lemma}
\newtheorem{remark}[theorem]{Remark}

\newtheorem{definition}[theorem]{Definition}
\newtheorem{cor}[theorem]{Corollary}

\begin{document}

\title{Uniruled symplectic divisors}

\author{Tian-Jun Li \& Yongbin Ruan}
\address{School  of Mathematics\\  University of Minnesota\\ Minneapolis, MN 55455}
\email{tjli@math.umn.edu}
\address{Department of Mathematics\\ University of Michigan\\ Ann Arbor,
MI}\email{ruan@umich.edu}

\maketitle \tableofcontents
\section{Introduction}

    This is the second in a series of papers devoted to the symplectic birational geometry
    program. A fundamental problem in symplectic geometry is to generalize
    birational geometry to symplectic category. Such a
    generalization should be viewed as the first step towards the
    classification of symplectic manifolds. Hopefully, it
    will provide a better understanding of birational geometry
    itself as well. In \cite{HLR}, the authors set up some general framework for such a symplectic
    birational geometry programs. Among other things, we proposed to use Guillemin-Sternberg's birational
    cobordisms
    to replace birational maps. Using sophisticated GW-machinery
      we also  settled successfully
     the  fundamental birational cobordism invariance of uniruledness.

     In symplectic category uniruledness is defined via
    Gromov-Witten invariants. More precisely,  a symplectic manifold is called uniruled if there is a
    non-vanishing genus zero GW-invariant $<[pt], \alpha_2, \cdots, \alpha_k>_{0,A}\neq 0$ for some nonzero
    class $A$.
     Such a class $A$ is referred as a uniruled class.

      Koll\'ar-Ruan showed that  all projective uniruled manifolds are
     symplectic uniruled.
    There are many reasons to believe that this is a better generalization of
    projective uniruledness
    to symplectic category than the geometric notion of ``a manifold covered by rational curve". However, many
    obvious properties of algebraic uniruledness are no longer obvious in our context. The birational cobordism invariance
     is
    such an example. In fact, it is related to a rather difficult problem in Gromov-Witten theory in terms of finding
    blow-up formula of Gromov-Witten invariants.

    The next step of symplectic birational geometry program is to
    study various surgery operations such as contraction, flip and
    flop.
   The main perspective comes from the basic fact in the
   projective birational program that
   various  birational surgery operations such as
   contraction and
    flop have a common feature: the subset being operated on
    is necessarily
    uniruled.
    Therefore in our program we also need to understand {\it uniruled symplectic  submanifolds}. In this article and its sequel, we
    focus on symplectic uniruled divisors.   Our key observation is that, as in the projective
    birational program,  such a divisor
    admits a dichotomy  depending on the positivity of its normal
    bundle. If the normal bundle is non-negative in certain sense, it
    will force the ambient manifold to be uniruled. If the normal
    bundle is negative in certain sense, we can contract it to
    obtain a simpler symplectic manifold. In this article, we
    treat the case of non-negative normal bundles.  In
    \cite{tLR} the negative case will be dealt with in
    dimension six.

 To state our main theorem  we need the notion of a
minimal uniruled class, which is a uniruled class with minimal
symplectic area among all uniruled classes.
     Suppose that $\iota: D\rightarrow (X, \omega)$
     is a  symplectic submanifold of codimension 2, i.e. a smooth divisor. Let $N_D$ be the normal bundle of
     $D$ in $X$. Notice that $N_D$ is a $2-$dimensional symplectic vector bundle and
     hence has a well defined first Chern class.  We will often  use $N_D$ to denote the first Chern class.

\begin{theorem} \label{main1} Suppose $D$ is uniruled and $A$ is a minimal uniruled class of $D$
such that
\begin{equation}\label{k}
<\iota^*\alpha_1, \cdots, \iota^*\alpha_l, [pt], \beta_2, \cdots,
\beta_k>^D_{0,A}\neq 0
\end{equation} for $k\leq N_D(A)+1$. Then $(X,\omega)$ is uniruled.
\end{theorem}

In particular, we have

\begin{cor} \label{homologically injective} Suppose $D$ is
uniruled and the normal bundle $N_D$ is non-negative on a minimal
uniruled class. Then $X$ is uniruled if either

\begin{itemize}
\item $D$ is homologically injective or
\item $D$ is projectively uniruled.
\end{itemize}
\end{cor}

Another consequence of Theorem \ref{main1} is

\begin{cor}
   Suppose that $X$ is a non-uniruled manifold containing a
   uniruled divisor satisfying (\ref{k}). Then, $k>N_D(A)+1$. In
   particular, $N_D(A)<0$ if $k=1$.
   \end{cor}

   The above corollary is the first step towards the constructions
   of symplectic divisoral contraction in the third paper of the
   series \cite{tLR}.

 The idea of proof is similar to that of
\cite{HLR}. We partition insertions of $D$ into
     two types, global insertions  $\iota^*\beta_i$ and local insertions $\alpha_j$. The
    relative/absolute correspondence of
    Maulik-Okounkov-Pandharipande  interchanges relative
    GW-invariants with certain admissible absolute GW-invariants having a similar partition of insertions.
    we extend the correspondence to include certain
    super-admissible GW-invariants. In addition, as in \cite{Ga},  in
    this case invariants of the divisor enters the ``extended"
    relative/absolute correspondence in a nontrivial way. In fact,
    this is our main strategy to lift a minimal uniruled invariant of divisor to
    a uniruled invariant of ambient manifold.

As a by product our main theorem also gives a rather general {\em
from divisor to ambient space} inductive construction of uniruled
symplectic manifolds.

     Theorem \ref{main1}
    can be easily applied in a variety of  situations, giving a comprehensive generalization of  some  early results of
    McDuff.
    For instance, if
    $D$ is a Fano manifold with  $b_2=1$ and non-negative normal bundle, then $X$ must
    be uniruled.  The case of  $D=\mathbb P^{n-1}$ is studied by McDuff in \cite{Mc1}, \cite{Mc2} (see section 2). We list many more
    examples in section 6. We should mention that another
    obvious
    inductive construction is {\em from fiber to total space}.
    In particular, a uniruled and homologically  injective fiber in a symplectic fibration will
    force the total space to be uniruled (Corollary  \ref{fibration}).
    Although it might be possible to derive this result as a consequence of our main theorem, it
    actually can be established by a simpler and classical argument.

    The paper is organized as follows. In section 2, we first review basic
    properties of uniruled symplectic manifolds. Then using direct geometric arguments we
    describe the fiber-to-total space approach, as well as  some  early
    examples of McDuff which motivate our divisor-to-ambient strategy.
    In section 3, we
    sketch a relative-divisor/absolute correspondence to connect absolute
    invariants with
    relative invariants and invariants of divisor. The new ingredient is the
    appearance of invariants of the divisor corresponding to the
    additional
    super-admissible absolute invariants.  To prove the main theorem, we need
    to study a more delicate version of our correspondence
    involving a point insertion (see section 3). This requires our extensive
    knowledge of relative invariants of $\mathbb P^1-$bundles. We establish
    it in  section 4 via several powerful techniques in GW theory. The main theorem is then proved in section 5. The applications
    will be studied in section 6.

    The authors would like to thank Dusa McDuff for her interest in this
    paper and for her numerous suggestions which corrected a number of  mistakes and greatly improved the presentation. We are also grateful to Davesh Maulik for
useful discussions.
    Both authors are supported by NSF.

 \section{Uniruled symplectic manifolds}

 \subsection{Definition and basic properties}

\begin{definition} Let $A\in H_2(X;{\mathbb Z})$ be a nonzero  class.
$A$ is said to be a uniruled class if there is a nonzero GW
invariant
\begin{equation}\label{invariant}\langle [pt], \alpha_2, \cdots ,
\alpha_k\rangle^X_A
\end{equation}
  with a point
insertion.

Let $A$ be a uniruled class.  $A$ is said to be a strongly uniruled
class if $k$ can be chosen to be $3$. $A$ is said to be minimal if
it has the smallest symplectic area among all uniruled classes.

\begin{remark}
 Clearly a
uniruled class $A$ is a spherical class with positive symplectic
area. Thus, in dimension 6 and higher, it follows from Gromov's
$h-$principle
 (see e.g. \cite{Ltj}) there is an
embedded symplectic sphere in the class $A$ passing through any
given point.  However, the converse is not true.

For the $4-$dimensional case see Proposition \ref{criterion}.

\end{remark}

\end{definition}

\begin{definition} $X$ is said to be (symplectically) uniruled if there is a uniruled
class, and strongly uniruled if there is a strongly uniruled class.
\end{definition}

The notion of strongly uniruled is studied e.g. in \cite{Lu1},
\cite{Lu2}, \cite{Lu3} and \cite{Mc3}.

\begin{remark}  If the insertions in (\ref{invariant})
are all of even degree, then the class is called an evenly uniruled
class. This notion is studied in \cite{Mc3} and there is provided a
beautiful characterizations of uniruledness in terms of  units of
the quantum cohomology ring.
\end{remark}

It is easy to see that  we could well use the more general
disconnected GW invariants to define this concept. Moreover, we have
the following basic property proved in \cite{HLR}.

\begin{theorem} \label{descendent} A symplectic manifold $X$ is uniruled if there is a nonzero
disconnected genus zero descendent  GW invariant involving a point
constraint.
\end{theorem}

This flexibility is important for the proof of another basic
property also proved in \cite{HLR}.

\begin{theorem} \label{blowup} Being uniruled is a birational cobordism property.
In particular, if $\tilde X$ is  a blow-up of $X$, then $X$ is
uniruled if and only $\tilde X$ is uniruled.
\end{theorem}

As mentioned in the introduction,  the notion of symplectically
uniruled is a natural extension of the fundamental notion in
algebraic geometry. More precisely,
 for projective
manifold, we have the following result of Koll\'ar and the second
author, further strengthened by McDuff.

\begin{theorem} \label{projective}A projectively
uniruled manifold is symplectically uniruled.  Moreover,  a minimal
uniruled class with respect to an ample class is strongly  uniruled
with both additional  insertion being  powers of a K\"ahler form.
\end{theorem}

\begin{proof}
Let $A$ be the class of a $\mathbb P^1$ of minimal energy through a
very general point $x_0\in X$.

Then    (cf. Theorem 4.2 in \cite{HLR})  for some $k$ there is a
nonzero invariant of the form
\begin{equation} \label{fixed marked points} <[pt] || [pt],\omega^{i_2},\dots,\omega^{i_k}>^X_A,
\end{equation}
where the first $[pt]$ represents the Poincar\'e dual of the point
class of $\overline{\mathcal M}_{0,k}$ and $\omega$ is a K\"ahler
form on $X$.

Choose a homogeneous basis of $H^2(X;\mathbb R)$
\begin{equation} \label{symplectic basis} \Upsilon=\{1, [\omega], \cdots, [\omega^n], e_{n+1},
\cdots\}, \end{equation}
 where $e_i\cdot e_j=0$ if $i\leq n$ and
$j>n$. This is possible as $[\omega^n]\ne 0$.

 Apply  the splitting axiom to the invariant (\ref{fixed marked
points}) with respect to the basis $\Upsilon$, (\ref{fixed marked
points}) is expressed as a sum of products of $3-$point invariants.
 in curve classes $A_1,\dots,A_r$.
One of the curve classes, say $A_1$, has a  $[pt]$ constraint. But
then $A_1$ must have a holomorphic representative through $x_0$.
Hence $\omega(A_1) = \omega(A)$ and all $A_j, j>1$ are zero.

By our choice of $\Upsilon$ it is easy to see that  the $k-$point
invariant (\ref{fixed marked points}) collapses to
 a nonzero
invariant of the form
\begin{equation}\label{pq}
I_{p,q} := <[pt],[\omega^p], [\omega^q>_A^D
\end{equation}
where $ p=\sum_{j\in J_1} i_j,q=\sum_{j\in J_2} i_j$  for some
partition $J_1,J_2$ of $\{2,\dots,k\}$.

\end{proof}

This sharper version due to McDuff is particularly powerful in light
of Theorem \ref{main1}, leading immediately to Corollary
\ref{homologically injective}.

In dimension 4 it follows from \cite{Mc1}, \cite{LL1}, \cite{LL2},
\cite{LM} that the converse of Theorem \ref{projective} is
essentially true.

Fano manifolds are (projectively) uniruled. The analogue of a Fano
     manifold in the symplectic category is a monotone symplectic
     manifold where $C_1=\lambda \omega$ for $\lambda>0$. It would
     be a challenging problem to show that  any monotone symplectic
     manifold is indeed uniruled.

Another important class of examples is provided by the following
beautiful result in \cite{Mc3}.

\begin{theorem} Hamiltonian $S^1-$manifolds are uniruled.
\end{theorem}

\begin{remark}\label{even}
It would be interesting to see whether there are uniruled
 manifolds such that every  uniruled  invariant must have  odd degree insertions. Interestingly, such
 a manifold can not be projective  by Theorem \ref{projective}. We
 are not aware of such manifolds.
 It seems that Hamiltonian $S^1-$manifolds are a good case to
 investigate.
\end{remark}

    A rich source of uniruled manifolds comes from
    uniruled  fibrations. Suppose that $\pi: X\rightarrow B$ is a fibration
    (with possibly singular fibers) where $X$ and $B$ are
    symplectic manifolds. We call it {\em an almost complex
    fibration} if there are tamed $J, J'$ for $X, B$ such that
    $\pi$ is almost complex. For example, by the famous Thurston construction,
    if a symplectic fiber
    bundle over a symplectic manifold has fiber $(F, \sigma)$ and $[\sigma]$ is
    a restriction class, then the total space $X$ has a symplectic form $\Omega$ that restricts to
    $\sigma$ on the  fibers and hence   is almost complex.
    Lefschetz fibrations, or more generally, locally holomorphic fibrations studied in \cite{Go2}
    are also almost complex.

    Let $\iota: \pi^{-1}(b)\rightarrow X$ be the embedding for a
    generic fiber over $b\in B$. We have the following

\begin{prop}\label{fiber}
    Suppose that $\pi: X\rightarrow B$ is an almost complex fibration between symplectic
    manifolds $X, B$. Then, for $A\in H_2(\pi^{-1}(b);\mathbb Z)$ and
    $\alpha_2,...,\alpha_k\in H^*(X;\mathbb R)$,
    \begin{equation}\label{=}
    <[pt], \iota^*\alpha_2, \cdots,
    \iota^*\alpha_k>^{\pi^{-1}(b)}_A=<[pt], \alpha_2, \cdots,
    \alpha_k>^X_A.
    \end{equation}

    \end{prop}

    \begin{proof}
    First of all, (\ref{=}) makes sense as both invariants are well-defined at the same
    time.
       Choose $J, J'$ such that $\pi$ is almost complex. Suppose
       that $f: C\rightarrow X$ is a genus 0 stable map with homology
       class $A$. Then, $\pi\circ f$ is holomorphic with zero
       homology class. Therefore, $im (\pi \circ f)$ is a point.
       Namely, $im (f)$ is contained in a fiber. Choose a point $pt \in
       \pi^{-1}(b)$. Then we have the identification of the moduli
       spaces of $k-$marked genus zero stable curves with the 1st
       marked point going to $pt$,
$$\overline{\M}^X_{0,k}(A, pt)=\overline{\M}^{\pi^{-1}(b)}_{0,k}(A, pt ).$$

Furthermore, they  have the same virtual fundamental cycles. As
$\pi$ is almost complex we have the splitting
$$f^*TX=f^*T\pi^{-1}(b)\oplus \underline \C^l,$$
where $l$ is the codimension of a fiber and $\underline \C^l$ is the
trivial complex bundle of dimension $l$. As $C$ has genus zero, we
have
$$H^1(C, \underline \C^l)=0.$$
 It implies that
$$[\overline{\M}^X_{0,k}(A, pt)]^{vir}=[\overline{\M}^{
\pi^{-1}(b)}_{0,k}(A, pt)]^{vir}.$$
    By integrating $\alpha_2, \cdots, \alpha_k$ against the
    virtual fundamental cycles, we obtain (\ref{=}).

\end{proof}

Consequently,  we have

\begin{cor}\label{fibration}
Suppose that $\pi: X\rightarrow B$ is an almost complex fibration
between symplectic
    manifolds $X, B$. If a smooth fiber is uniruled and
    homologically
    injective (over $\mathbb R$), then $X$ is uniruled.
\end{cor}

The homologically injective assumption could be a strong one.
 Notice that for a fiber bundle,  the Leray-Hirsch
theorem asserts that, under the  homologically injective assumption,
the homology group of the total space is actually isomorphic to the
product of the homology group of the fiber and the base. However,
Corollary \ref{fibration} can still be applied for all  product
bundles, and
 all projective space fibrations (more generally, if the rational
cohomology ring of a smooth uniruled fiber is generated by
   the restriction of $[\omega]$).

Moreover, we were informed by McDuff that a Hamiltonian bundle is
homologically injective (or equivalently, cohomologically split) if
(cf. \cite{LM2})

a)  the base is $S^2$ (Lalonde-McDuff-Polterovich), and more
generally, a complex blow up of a product of projective spaces,

b) the fiber satisfies the hard Lefschetz condition (Blanchard), or
its real cohomology is generated by $H^2$.

 Here is another variation. As in the case of a projective space,
for a uniruled manifold up to dimension $4$,  insertions of a
uniruled class can all be assumed to be of the form $[\omega]^i$,
thus we also have

\begin{cor} If the general fibers of a possibly singular uniruled
fibration are $2$-dimensional or $4-$dimensional, then the total
space is uniruled.
\end{cor}

This in particular applies to a $2-$dimensional  symplectic conic
bundle. A symplectic conic bundle  is a  conic hypersurface bundle
in a smooth $\mathbb P^k$ bundle. Holomorphic conic bundles are
especially important in the theory of $3-$folds. It is conjectured
that a projective uniruled $3-$fold is either birational to a
trivial $\mathbb P^1-$bundle or a conic bundle.

\subsection{Some motivating examples}

    In this subsection, we present examples of uniruled manifolds from the divisor-to-ambient construction.
 These
    examples motivate Theorem \ref{main1}  and generalize
     some  early results of McDuff
     in a slightly different context.
The common feature is that the geometric situation is simple
    enough that a direct geometric argument can be applied.

    We start
    from the simplest situation of trivial normal bundles.
    Let $\iota:D\rightarrow X$ be a symplectic divisor.
    McDuff treated the case that $D$ is a standard projective space and $X$ is semi-positive, see
    Theorem \ref{McDuff}. In general we have,

    \begin{theorem}\label{trivial normal bundle}
Suppose the normal bundle  $N_D:=N_{D|X}$ is trivial. If there is a
nonzero invariant $\langle [pt], \iota^*\alpha_2, ...,
\iota^*\alpha_k\rangle_A^D$,
  then $X$ is uniruled and in fact,
  $$\langle [pt], \alpha_2, \cdots, \alpha_k\rangle_A^X=\langle [pt], \iota^*\alpha_2, ...,
  \iota^*\alpha_k\rangle_A^D$$
\end{theorem}

\begin{proof} The argument is parallel to that of Proposition
\ref{fiber}.
   First
notice that the triviality of the normal bundle implies that the
    $$vir\dim \overline{\M}^X_{0,k}(A)=vir\dim \overline{\M}^D_{0,k}(A)+2.$$
On the other hand, $\deg_X([pt])=\deg_D([pt])+2$. Hence $\langle
[pt], \tilde \alpha_2,..., \tilde \alpha_k\rangle _A^X$ is also
defined.

We choose an $\omega-$compatible almost structure $j$ on $D$ and
extend it to    an $\omega-$compatible almost complex structure on
$X$. Furthermore, we choose $J$ in such a fashion that a
 $D$ has an almost complex product neighborhood.

Choose a point $pt\in D$. Let $\overline{\M}^X_{0,k}(A, pt)$ and
$\overline{\M}^D_{0,k}(A, pt)$ be the moduli spaces of genus zero
stable maps of homology class $A$ such that $e(x_1)=pt$. Suppose
that $f: C\rightarrow X$ is a genus zero stable map in
$\overline{\M}^X_{0,k}(A, pt)$. It is well-known that any component
of $im(f)$ either lies in $D$ or intersects $D$ positively. One the
other hand, $D\cdot A=0$ by the assumption and $f(x_1)\in D$.
Therefore, $im(f)$ lies completely inside  $D$. Namely,
$\overline{\M}^X_{0,k}(A, pt)=\overline{\M}^D_{0,k}(A, pt)$.

Furthermore, $f^*TX=f^*TD\oplus \underline\C$ and $H^1(C, \underline
\C)=0$  imply that $$[\overline{\M}^X_{0,k}(A,
pt)]^{vir}=[\overline{\M}^D_{0,k}(A, pt)]^{vir}.$$
    By integrating $\alpha_2, \cdots, \alpha_k$ against the
    virtual fundamental cycles, we obtain
  $$\langle [pt], \alpha_2, ..., \alpha_k\rangle _A^X=\langle [pt], \iota^*\alpha_2,..., \iota^*\alpha_k\rangle_A^D.$$
Therefore $X$ is uniruled.

\end{proof}

     When the normal bundle is not trivial, the situation is more
     complicated. But the identification between appropriate GW-invariants of
     the divisor and  the ambient manifold still remains to be
     valid in some cases. The following is a particular important example in \cite{Mc2}
     established in the early 90s.

     \begin{theorem}  \label{McDuff} Let $(X, \omega)$ be a
     semi-positive symplectic $2n-$manifold which contains a submanifold $P$
     symplectomorphic to $\mathbb P^{n-1}$ whose normal
     Chern number $m$ is non-negative. Then certain blow-up of $X$ is
     uniruled,
     and if $0\leq m \leq 2$ or $n=2$, $X$ itself is uniruled.

     \end{theorem}

     As a consequence of Theorem \ref{blowup},
$X$ itself is
     still uniruled even if $n>1$ and $m\geq 3$.

The case of $D=\mathbb P^1$ was first proved in \cite{Mc1},
generalizing a result of \cite{Gr}. Moreover, in that case, $X$ is
shown to be the connected sum  of either $\mathbb P^2$ or an
$S^2-$bundle over a Riemann surface with a number of
$\overline{\mathbb P}^2$. This is a fundamental result  in
symplectic $4-$manifold theory.

It is instructive to examine her argument in the case of $D=\mathbb
P^2$ geometrically.

\begin{remark} Notice that here we are in the semi-positive territory, thus we can
directly compute invariants via cutting down a generic moduli space
by generic cycles. We will not mention this explicitly.
\end{remark}

 Let us first consider the case of trivial normal bundle.
 We pick a point $x$ in $D$ and a surface  $F$ intersecting
$D$ at 1 point $y \ne x$. Let $l$ be the line class. By the
previous proposition,
$$<[x], F>^X_l=<[x], [y]>^D_l=1.$$

In the case of normal bundle $\mathcal O(1)$, we choose two points
$x$ and $y$ in $M$. Any $A-$curve outside $D$ can intersect $D$ only
at one point, and there is a unique line in $D$ through $x$ and $y$.
Namely,
$$<[x], [y]>^X_l=<[x],[y]>^D_l=1.$$

In the case of normal bundle $\mathcal O(2)$, we choose two points
$x$ and $y$ and a line $L$ in $D$ away from $x$ and $y$. Namely,
$$<[x], [y], [L]>^X_l=<[x], [y], [L]>^D_l=1.$$

In each of above cases, we show that the GW-invariant of the ambient
manifold $X$ is equal to the corresponding invariant of $D$.

The next case of normal bundle $\mathcal O(3)$ is different. The
simple relation between Gromov-Witten invariant is no longer true.
McDuff's strategy is to blow up a line of $\mathbb P^2$ in $X$ to
reduce it to the previous situation.  Here, we give a different
argument which motivates the correspondence in the next section.

Now, for the invariant of $X$ to be well defined we need two points
and two lines, or 3 points, or one point and four lines. We choose
two points $x$ and $y$. We also pick two lines $L_1$ and $L_2$ in
$D$. Let $z$ be the intersection point of $L_1$ and $L_2$. We assume
that $z$ is not in the unique line through $x$ and $y$.
 We claim that

 \begin{equation} \label{combination}1=<[x], [y]>^D_l=<[x], [y], [L_1], [L_2]>^X_l-<[x],[y],[z]>^X_l.
 \end{equation}

  The point is that any curve through $x,y,z$ also
intersects $(x,y, L_1, L_2)$. Let $C$ be  a curve intersecting $(x,
y, L_1, L_2)$. If $C$ is not inside $D$, then $C$ has to intersect
$z$ because $C$ has at most three intersection points with $D$. If
$C$ is inside $D$, then $C$ must be the unique line through $x$ and
$y$.

We remark that the above proof  is just a sketch. To do this
calculation we realize the constraints $x, y. L_1, L_2$ in a very
non-generic way, and for a complete proof we would have to prove
that this is justified.

It follows  that  one of the invariants on the right hand side of
equation (\ref{combination})is not zero and hence $X$ is uniruled.
We want to emphasis that in this case the invariant of $D$ can not
be identified with a single invariant of $X$.

    When $\mathcal O(m)$ increases, we can similarly express
    the relevant  invariant of $D$ as a more and more complicated combination
    of invariants of the ambient space $X$. Our main idea is that such
    a process is best cast into the language of the
    relative-divisor/absolute correspondence established in
    the following 3 sections.

\section{Degeneration formula and correspondence}

      A powerful tool in GW theory is the degeneration formula.
      To explore its power systematically, a very useful "relative/absolute correspondence"
      was constructed by
     Maulik-Okounkov-Pandharipande. It has been generalized to the
     situation of blow-up by the authors and Hu to relate absolute
     invariants of a manifold and relative invariants of the blow-up
     manifold. Such an extended  relative/absolute correspondence is crucial
     to prove the birational invariance of uniruledness.

     However, only a subset set of colored absolute invariants appears in the
     relative/absolute correspondence.  They are admissible in the sense that the multiplicity
     of relative insertions is exactly $D\cdot A$, where $D$ is
     the divisor and $A$ is the curve class.
     It is natural to consider non-admissible absolute invariants.
     If the absolute invariant is sub-admissible in the sense that
     the multiplicity of relative insertions is less than $D\cdot A$,
     we can always use the divisor axiom to add more insertions to
     obtain an equivalent admissible invariants. The interest is
     on  super-admissible invariants where the multiplicity is
     bigger than $D\cdot A$.

     \subsection{Symplectic
cut and the degeneration formula}\label{cut}

\subsubsection{Symplectic cut} Let $(X,\omega)$ be a closed symplectic manifold. Let
$S$ be a hypersurface having a neighborhood with a free Hamiltonian
$S^1-$action. For instance, if there is a symplectic submanifold in
$X$, then hypersurfaces corresponding to sphere bundles of the
normal bundle have this property. Let $Z$ be the symplectic
reduction at the level $S$, then $Z$ is the $S^1-$quotient of $S$
and is a symplectic manifold of 2 dimension less.

We can cut $X$ along $S$ to obtain two closed symplectic manifolds
$(\overline X^+,\omega^+)$ and $(\overline X^-,\omega^-)$ each
containing a smooth copy of $Z$, and satisfying $\omega^+\mid_Z =
\omega^-\mid_Z$ (\cite{Le}).

 In particular, the pair $(\omega^+,
\omega^-)$ defines a cohomology class of
$\overline{X}^+\cup_Z\overline{X}^-$, denoted by
$[\omega^+\cup_Z\omega^-]$.
 Let $p$ be the continuous collapsing map
$$p:X\to \overline{X}^+\cup_Z\overline{X}^-.$$
It is easy to observe that
\begin{equation}\label{cohomology relation}
   p^* ([\omega^+\cup_Z\omega^-]) = [\omega].
\end{equation}

\subsubsection{Degeneration formula}\label{df}
 Given a symplectic cut, there is a  basic link between absolute invariants
 of $X$ and relative invariants
of $(\overline{X}^{\pm}, Z)$ in \cite{LR} (see also \cite{IP},
\cite{Li2}).  We now  describe such a formula.

 Let $B\in
H_2(X;{\mathbb Z})$ be in the kernel of
$$
  p_* : H_2(X;{\mathbb Z})
\longrightarrow H_2(\overline{X}^+\cup_Z\overline{X}^-; {\mathbb
Z}). $$
 By (\ref{cohomology relation}) we have $\omega(B) =0$.
Such a class is called a vanishing cycle.
 For $A\in H_2(X; {\mathbb Z})$ define $[A] = A + \mbox{Ker}
(p_*)$ and
\begin{equation}\label{vanishing cycle}
\langle\tau_{d_1}\alpha_1, \cdots,
\tau_{d_k}\alpha_k\rangle^X_{g,[A]} =
\sum_{B\in[A]}\langle\tau_{d_1}\alpha_1, \cdots,
\tau_{d_k}\alpha_k\rangle^X_{g,B}.
\end{equation}
Notice that  $\omega$ has constant pairing with any element in
$[A]$.
 It follows from  the Gromov
compactness theorem that there are only finitely many such elements
in $[A]$ represented by $J$-holomorphic stable maps. Therefore, the
summation in (\ref{vanishing cycle}) is finite.

At this stage  we need to assume that each cohomology class
$\alpha_i$ is of the form \begin{equation}\label{distribution}
\alpha_i = p^*(\alpha_i^+\cup_Z\alpha_i^-).
\end{equation}
 Here $\alpha_i^\pm \in H^*(\overline{X}^\pm; {\mathbb
R})$ are classes with  $\alpha_i^+\mid_Z = \alpha_i^-\mid_Z$ so that
 they give rise to a class $\alpha_i^+\cup_Z\alpha_i^-\in
H^*(\overline{X}^+\cup_Z\overline{X}^-; {\mathbb R})$.

The degeneration formula expresses $\langle\tau_{d_1}\alpha_1,
\cdots, \tau_{d_k}\alpha_k\rangle^X_{g,[A]} $ as a sum of products
of relative invariants of $(\overline{X}^+, Z)$ and
$(\overline{X}^-, Z)$, possibly with disconnected domains.  In each
product of relative invariants, what is relevant for us are the
following conditions:

$\bullet$ the union of two domains  along relative marked points is
a stable genus $g$ curve with $k$ marked points,

$\bullet$ the total curve class  is equal to $p_*(A)$,

$\bullet$ the relative insertions are dual to each other,

$\bullet$ if $\alpha_i^+$ appears for $i$ in a subset of
$\{1,\cdots, k\}$, then $\alpha_j^-$ appears for $j$ in the
complementary subset of $\{1,\cdots, k\}$.

\subsection{Sup-admissible graphs}

Let $\iota:D\to X$ be a smooth connected symplectic divisor. As
mentioned, we can cut along $D$, or precisely, a cut along a small
circle bundle $S$  over $D$ inside $X$.

In this case, as a smooth manifold, $\overline{X}^+ =X$, which we
will denote by $\tilde X$. Denote the symplectic reduction of $S$ in
$\tilde X$ still by $D$. Notice however, the symplectic structure is
different from the original divisor. And $\overline{X}^-= \mathbb
P(N_D\oplus \underline {\mathbb C})$, the projectivization of
$\mathbb P(N_D\oplus \underline {\mathbb C})$\footnote{Notice that
our convention here is opposite to that in \cite{HLR}}. We will
often denote it simply by $P_D$ or $P$. Notice that $\mathbb
P(N_D\oplus \underline {\mathbb C})$ has two natural sections,
$$  D_{0}=\mathbb P(0\oplus \underline \C), \quad D_{\infty}=\mathbb P(N_D\oplus 0).$$
 The symplectic reduction of $S$ in $P_D$ is the
section $D_{\infty}$.

In summary, in this case,
 $X$ degenerates into $(\tilde X, D)$ and
$(P_D, D_{\infty})$. We also denote $\omega^-$ by $\omega_P$.

We also observe that the section $D_{0}$ actually has the same
symplectic structure and same neighborhood as the original divisor.
We denote the inclusion of $D_{0}$ in $P_D$ still by $\iota$.

\begin{definition}\label{effective}
 A class $A\in H_2( X;\mathbb Z)$ is called effective for the symplectic cut
 along $D$ if either

 $A$ is represented by a
 pseudo-holomorphic stable map to $X$ for
 all $\omega-$tamed almost complex structures, or

 $A$ is represented by a
 pseudo-holomorphic stable map to $X^-$ for
 all $\omega_P-$tamed almost complex structures, or

  $A$ is in the image of $\iota_*$ and is represented by a
 pseudo-holomorphic stable map to $\tilde X$ for
 all $\omega^+-$tamed almost complex structures.

\end{definition}

Notice that the zero class $A=0$ is considered to be effective here
as a constant map is pseudo-holomorphic.

\begin{definition} \label{absolute graph}
       A  connected colored graph $\Gamma$ consists of one vertex decorated by
        $(g,A)$ with $g$ an integer, $A$ an  class in $H_2(X;\mathbb Z)$, and   two sets of colored tails, $X$-tails and $D$-tails.

        We  further weight
        each  $X$-tail by a class $\alpha_i\in H^*(X;\mathbb R)$, called an
        $X$-insertion. We also weight
         each $D$-tail by a pair
$(\mu_i, \beta_i)$, where $\mu_i$ is a non-negative integer, and
$\beta_i$ is a class in $H^*(D;\mathbb R)$ called a $D$-insertion.
        We call the  resulting graph  {\em a connected colored weighted
        graph} and denoted by
        $$\Gamma((\alpha_i)| ((\mu_i, \beta_i))).$$
        The collection of pairs, $\mu=((\mu_i, \beta_i))$, is called a
        weighted partition.

         There is also the distinguished graph, the empty
        graph $\Gamma(\emptyset|\emptyset)$.
 \end{definition}

 \begin{definition}
 The dimension of the empty graph is defined to be zero. For a
 nonempty graph $\Gamma((\alpha_i)| ((\mu_i, \beta_i)))$, its dimension
 is defined to be
\begin{equation}\label{dimension of graph}\begin{array}
  {ll}
 &\dim \Gamma((\alpha_i)| ((\mu_i, \beta_i)))\cr
 =&2[C_1(A)+(n-3)(1-g)+D\cdot A]\cr &+[\sum (2-2\mu_i-\deg(\mu_i)]\cr
&+\|\varpi\|_1\cr
 &+[\sum_{\deg(\alpha_i)\ne 1}(2-\deg(\alpha_i)],\cr
\end{array}
\end{equation}
where $\|\varpi\|_1$ is the number of degree 1 insertions in
$\varpi$.
\end{definition}

\begin{remark} We can also consider the
       disjoint union $\Gamma^{\bullet}$ of several such graphs  and use $A_{\Gamma^{\bullet}},
       g_{\Gamma^{\bullet}}$ to denote the total homology class and total arithmetic genus.
       Here the total arithmetic genus is $1+\sum (g_i-1)$.
\end{remark}

\begin{definition}A connected colored  weighted graph with
$$ \varpi=(\alpha_i) \quad \hbox{and} \quad  \mu=((\mu_i, \beta_i)),$$
written simply as $\Gamma(\varpi|\mu)$, is called
\begin{itemize}
\item admissible if $\sum_j \mu_j =
D\cdot A$,
\item strictly sup-admissible if
\begin{equation}\label {sup-ad}\sum_j \mu_j> D\cdot A
\quad \hbox{and} \quad A\in im[\iota_*:H_2(D;\mathbb Z)\to
H_2(X;\mathbb Z)],
\end{equation}
\item strictly sub-admissible if $\sum_j
\mu_j< D\cdot A$.
\end{itemize}
A possibly disconnected graph is  called sup-admissible if it is a
connected strictly sup-admissible graphs or the disjoint union of
one or more connected  admissible graphs.

\end{definition}

Notice that every strictly sub-admissible absolute invariant can be
made admissible by adding an appropriate  number of $D$-insertions.

These graphs will be used to describe the structure of the
components appearing in the decomposition formula; cf. equation
(\ref{lower2}). The strictly sup-admissible graphs are connected
because they correspond to curves that lie entirely in $P_D$. The
other sup-admissible graphs describe the part of the curve lying in
$\tilde X$ and hence need not be connected.

\begin{definition}\label{fix bases}
 Suppose $X$ is of dimension $2n\geq 4$. Let
$\Theta=\{\delta_i\}$ be a self dual basis of
$\oplus_{q=0}^{2n-2}H^q(D;{\mathbb R})$ with respect to the cup
product of $D$.

 Let $\Xi=\{\gamma_i\}$ be a basis of $\oplus_{0\leq p\leq 2n}H^{p}(X;{\mathbb R})$.
\end{definition}

We will fix $\Theta$ and $\Xi$ in the rest of this paper.

\begin{remark} Notice  that we do not
require any compatibility of $\Theta$ and $\Xi$.
\end{remark}


 \begin{definition}
An $\Theta-$standard
  weighted partition $\mu$ is a partition
weighted by classes of $D$ from $\Theta$, i.e.
$$
   \mu = \{ (\mu_1, \delta_{K_1}), \cdots, (\mu_{\ell(\mu)},
   \delta_{K_{\ell(\mu)}})\}.
$$
$\varpi=(\alpha_i)$ is called $\Xi-$standard if each $\alpha_i\in
\Xi$.
\end{definition}

Let $c(X, \omega, J)$ be the minimal symplectic area of a connected
non-constant $J-$holomorphic curve. $c(X, \omega, J)$ is positive
due to Gromov compactness.

Let $c(X, \omega)$ be the maximum of $c(X, \omega, J)$ over $J$.

\begin{definition}$\Gamma(\varpi| \mu)$ is
called standard if
\begin{itemize}
\item the class of each vertex is a nonzero effective class,
\item $g\geq -\frac{\omega(A)}{c(X,\omega)}+1$.
\item  $\varpi$ is
$\Xi-$standard,
\item
 $\mu$ is $\Theta-$standard,

\item its dimension is zero.
\end{itemize}
\end{definition}

\subsubsection{Ordering  the graphs}
 Let $I$ be the set of possibly disconnected  sup-admissible standard colored weighted
 graphs.
We will order $I$  following \cite{MP}. The partial order is
defined in terms of several preliminary partial orders.

\begin{definition}
The set of pairs $(m, \delta)$ where $m\in {\mathbb Z}_{>0}$ and
$\delta \in H^*(D;\mathbb R)$ is partially ordered by the following
size relation
\begin{equation}\label{size}
   (m, \delta) > (m', \delta')
\end{equation}
if $m>m'$ or if $m=m'$ and $\deg (\delta) > \deg (\delta')$.
\end{definition}

We may place the pairs of $\mu$ in decreasing order by size, i.e. by
(\ref{size}).

\begin{definition}
A {\it lexicographic } ordering on weighted partitions is then
defined as follows:
$$
   \mu \stackrel{l}{>} \mu'
$$
if, after placing $\mu$ and $\mu'$ in decreasing order by size,
the first pair for which $\mu$ and $\mu'$ differ in size is larger
for $\mu$.
\end{definition}

 Next we introduce a relevant partial order on the
effective curve classes  of $X$ (see Definition \ref{effective}).

\begin{definition}

 For effective classes $A$ and $ A'$ in $ H_2( X;{\mathbb Z})$, we say that
$A'<A$ if  $A-A'\in H_2(X;\mathbb Z)$ has positive pairing with the
symplectic form on $X$.
\end{definition}


 We partially order such weighted graphs
 in the following way.

 \begin{definition}\label{order} The empty graph is smaller than
 any other graph. For any two non-empty  admissible  graphs
 $
 \Gamma(\varpi'|\mu')$ and $  \Gamma(\varpi|\mu) $,
$$
 \Gamma(\varpi'|\mu')\quad
\stackrel{\circ}{<}\quad  \Gamma(\varpi|\mu)
$$
if one of the conditions below holds

(1) $A'<A$,

(2) equality in (1) and the arithmetic genus satisfies $g'<g$,

 (3) equality in (1-2) and $\|\varpi'\| < \|\varpi\|$,

(4) equality in (1-3) and $\deg (\mu') > \deg (\mu)$,

(5) equality in (1-4) and $\mu'\stackrel{l}{>} \mu$,

\noindent where  $\|\varpi\|$  denotes  the number of $X$-tails, and
 $\deg(\mu)$ is the sum of
$\deg(\mu_i)$.

If $ \Gamma(\varpi'|\mu')$ is  admissible and $  \Gamma(\varpi|\mu)$
 is connected and strictly sup-admissible,
$$
 \Gamma(\varpi'|\mu')\quad
\stackrel{\circ}{<}\quad  \Gamma(\varpi|\mu)
$$
if $A'\leq A$.

\end{definition}

The inequalities (3-5) are designed so that the dimension of the
moduli space satisfying the larger constraint/condition is larger.
This explains the seemingly strange conditions (4) and (5)
 where the inequalities are reversed.

\begin{remark}\label{multiplicationorder}
 It is easy  to observe
that this extended partial order $\stackrel{\circ}{<}$ is preserved
under disjoint union of admissible graphs. Notice that we don't
compare strictly sup-admissible graphs.
\end{remark}

\begin{remark}If we are only interested in genus zero invariants, then
we can replace $g'<g$ in  (2) by the inequality of the number of
connected components, $n'>n$.
\end{remark}

Here is a crucial property of the ordering.

\begin{lemma} \label{bounded} Given a standard colored weighted graph
 there
are only finitely many standard colored weighted lower in the
partial ordering. In particular, there is a minimal standard
invariant with $A\ne 0$ and nonzero value.
\end{lemma}

\begin{proof}
As a strictly sup-admissible graph is not smaller than any other
graph, we only need to bound the number of admissible graphs.

First of all, the number of effective classes with area bounded
above is finite due to the Gromov compactness.

In particular, there is a minimal area $c(X, \omega)$ among all
nonzero effective classes. As each vertex is a nonzero effective
class, this implies the number of components of a standard graph
with bounded area is bounded by $g\geq
\frac{\omega(A)}{c(X,\omega)}$. Hence the total genus is bounded
from below by $g\geq -\frac{\omega(A)}{c(X,\omega)}+1$.

 As the number of $\varpi$ insertions is bounded from above and the $\varpi$ insertions are
chosen from a finite generating set $\Xi$, there is only a finite
number of choices of $\varpi$.



Finally,   the number of $\mu$ insertions and the total multiplicity
of $\mu$ are both  bounded by the intersection number $D\cdot A$. As
the $\mu$ insertions are chosen from a finite generating set
$\Theta$ and the multiplicities are positive, there is only a finite
number of choices of $\mu$.

\end{proof}

A partially order set is called  {\em lower
bounded} if there are only finitely many elements lower than a given element.
$I$ is lower bounded, so is any subset of $I$.

 \subsection{Invariants associated to graphs}
In this subsection, we associate to  a sup-admissible standard
colored weighted graph certain GW invariants of the symplectic cut.
We just give the definition for connected graphs, the extension to
disconnected graphs
 is straightforward.

\subsubsection{Absolute invariants}

\begin{definition}\label{tilde}
For a relative insertion $(m, \delta)$, we associate the absolute
descendent insertion $\tau_{m-1}(\tilde \delta)$ on $X$ supported on
$D$, where
 $\tilde \delta_i=\delta_i [D]$.
Given a standard (relative) weighted partition $\mu$,
 let
\begin{equation}\label{d}
d_i(\mu)=\mu_i-1,
\end{equation} and
\begin{equation}\label{tildemu}
\tilde \mu=\{\tau_{d_1(\mu)}(\tilde \delta_{K_1}), \cdots ,
\tau_{d_{l(\mu)}(\mu)}(\tilde \delta_{K_{l(\mu)}}) \}.
\end{equation}
\end{definition}

It is convenient to view $[D]$ as the class of a Thom form supported
near
 the symplectic divisor $D$.
 Then class $\tilde \delta=\delta[D]$ is the represented by the
 wedge product of the pull back of a form
    representing  $\delta$ in a neighborhood of $D$ with the compactly supported Thom form  of $D$.
 In terms of homology constraints,
 $\tilde \delta$ and $\delta$ correspond to the same cycle
 lying inside $D$.

\begin{definition}\label{associate}
The absolute descendent invariant associated to a connected standard colored
weighted graph $\Gamma(\varpi|\mu)$ is
$$ \langle {\Gamma}(
\varpi; \tilde \mu)\rangle^{X}.$$ The invariant associated to the
empty graph is called the empty absolute invariant and its value
is defined to be $1$.
\end{definition}

    Notice for such an absolute descendent invariant of $X$ all the descendent insertions
are supported on $D$. Such an invariant is  colored in the sense
that
 the insertions are  divided into two collections, the
 $X$-insertions
 $\varpi$ and the $D$-insertions $\tilde \mu$, with each insertion
in $\varpi$ being of the form $\gamma_L$,  and  each insertion in
$\tilde \mu$ being of the form $\tau_d \tilde \delta_K$.

\subsubsection{Relative invariants}

\begin{definition}\label{associate2}
Let  $\Gamma(\varpi|\mu)$ be a connected standard colored weighted
graph.

 If it is admissible, the relative invariant of
the symplectic cut
 associated to it is
 $$\langle\Gamma(\varpi|\mu)\rangle^{\tilde X, D}.$$

 If it is strictly sup-admissible, the
relative invariant of the symplectic cut
 associated to it is
$$\langle  \pi^*\iota^*\varpi, \tilde \mu| \emptyset\rangle^{P,
D_{\infty}}_A,$$ where we view $P$ as a bundle over its zero section
$D_0$ and $\pi:P\to D_{0}$ is the projection, and $\tilde \mu$ here
is given by
$$(\tau_{d_1(\mu)}(\delta_{K_1} [D_{0}]),\cdots,
\tau_{d_{l(\mu)}(\mu)}(\delta_{K_{l(\mu)}}[D_{0}])).$$

Finally, the invariant associated to the empty graph is the empty
relative invariant and its value is defined to be $1$.
\end{definition}

\subsection{Sup-admissible correspondence}

 Consider the
 infinite dimensional vector space ${\mathbb R}^I$ whose coordinates
 are ordered in the way compatible with the partial order of $I$.
  From the  relative invariants in Definition \ref{associate2} we can form a vector
 $$v_{rel}\in {\mathbb R}^I$$
 given by the numerical values.
 We also have
the   vector
$$ v_{abs}\in {\mathbb R}^I$$
 given by the numerical values of  the sup-admissible invariants of $X$ relative to
$D$ in definition \ref{associate}.

\begin{theorem} \label{generalization}
There is an invertible lower triangular linear transformation
$$A: {\mathbb
R}^I\rightarrow {\mathbb R}^I$$  such that (i) the coefficients of
$A$ are  local in the sense of being dependent on $D$ only; (ii)
$$A(v_{rel})=v_{abs}.$$
In particular, $v_{rel}$ and  $v_{abs}$ determine each other.

Finally, if $I_{0}\subset I$ denotes the subset of genus zero
invariants with $\varpi=\emptyset$, then $A$ further restricts to
an invertible lower triangular transformation from  ${\mathbb
R}^{I_{0}}$ to ${\mathbb R}^{I_{0}}$.
\end{theorem}

\begin{proof} The idea is as follows.
  We start with  a connected colored
weighted graph $\Gamma( \varpi |\mu)$. The associated absolute
invariant is $ \langle {\Gamma}( \varpi; \tilde \mu)\rangle^X$.

 We apply the degeneration formula to this connected
absolute invariant  to express it  as a linear combination of
relative invariants of $(\tilde{X}, D)$ with the coefficients being
essentially certain relative invariants of the $\mathbb P^1-$bundle.
In the strictly sup-admissible case there is an additional term
being the associated relative invariant of $(P, D_{\infty})$.

This is possible because the homomorphism $p_*$ is obviously
injective.

Of course we also need to first split the $\varpi$ and $\tilde \mu$
insertions as in (\ref{distribution}). Recall that each $\mu$
insertion is of the form $\gamma=\tau_d(\delta[D])$ for $\delta\in
\Theta$.   Then we set
$$\gamma^+=0, \quad \gamma^-=\tau_d(\delta[D_0]).$$
In other words we distribute all the $\tilde \mu$ insertions to the
${\mathbb P}^1-$bundle side. With this preferred distribution of
insertions, the original graph $\Gamma(\varpi|\mu)$ turns out to be
the largest weighted relative graph appearing in the linear
combination. For an $\varpi$ insertion $\tau$ we can take the $+$
class to be itself and
 the $-$ part to be the class  of the $\mathbb
P^1-$bundle over the cycle of intersection, i.e
$$\tau^+=\tau,\quad \tau^-=\pi^*\iota^*\tau.$$

The arguments for $I_0$ and $I$ are similar, we just treat the case
of $I_0$, i.e. genus 0 and $\varpi=\emptyset$.

 The absolute invariant associated to $\Gamma(\emptyset|\mu)\in I_0$ is of the form
\begin{equation}\label{admissible}
\langle \tau_{d_1(\mu)}(\tilde \delta_{K_1})),\cdots,
\tau_{d_{l(\mu)}(\mu)}(\tilde
\delta_{K_{l(\mu)}})\rangle^X_B.\end{equation}

With all insertions  distributed to the $\mathbb P^1-$bundle side,
 (\ref{admissible}) is  expressed as the following sum

\begin{equation}\label{lower2}
\begin{array}{ll}
&\sum \langle {\Gamma^-(\emptyset|\eta}) \rangle^{\tilde X,
D}\Delta(\eta) \langle \Gamma^+(\tau_{d_1(\mu)}(
\delta_{K_1}[D_0]),\cdots, \tau_{d_{l(\mu)}(\mu)}(\tilde
\delta_{K_{l(\mu)}}[D_0]) |\breve \eta) \rangle^{P,D_{\infty}}
\end{array}
\end{equation}
 over  appropriate pairs of weighted graphs. Here $\Delta(\eta)$
 is a nonzero combinatorial constant depending on the multiplicities of $\eta$.

If the graph $\Gamma(\emptyset|\mu)$ is strictly sup-admissible,
i.e. $\sum_j \mu_j> D\cdot A$ and  $A$ is in the image of
$\iota_*:H_2(D;\mathbb Z)\to  H_2(X;\mathbb Z)$, then there is a
term with $\breve \eta=\eta=\emptyset$ in (\ref{lower2}). In this
term the relative invariant of $(P, D_{\infty})$ is the relative
invariant associated   to the given graph in $I_0$, and the relative
invariant of $(\tilde X, D)$ is associated to the empty graph and
hence value equal to 1.

In any  other term with $\eta\ne \emptyset$ we have a relative
invariant of $(\tilde X, D)$ associated to a possibly disconnected
admissible graph $\Gamma^-(\emptyset|\eta)\in I_0$. Regard  the
relative invariant
$$\langle
\Gamma^+(\tau_{d_1(\mu)}( \delta_{K_1}[D_0]),\cdots,
\tau_{d_{l(\mu)}(\mu)}(\tilde \delta_{K_{l(\mu)}}[D_0]) |\breve
\eta) \rangle^{P, D_{\infty}}$$ of $(P, D_{\infty})$ as the
coefficient of the graph of $\Gamma^-(\emptyset|\eta)$. The
coefficient is nonzero only if
  the class of $\Gamma^-$ is at most $A$. Thus  we have
 $\Gamma^-(\emptyset|\eta) \stackrel{\circ}{<}
\Gamma(\emptyset|\mu)$, according to our extended order in
 Definition \ref{order}.

Suppose the graph $\Gamma(\emptyset|\mu)$ is admissible. For the
term with $\eta=\emptyset$, the relative invariant of $(P,
D_{\infty})$ is not associated to any graph in $I_0$ as
$\Gamma(\emptyset|\mu)$ is not strictly sup-admissible. Instead the
relative invariant of $(P, D_{\infty})$ is considered to be the
coefficient of the empty graph in $I_0$. But empty graph is
certainly smaller than $\Gamma(\emptyset|\mu)$ according to
Definition \ref{order}.
 For all other terms, as our order of admissible graphs agrees with
 that in \cite{MP},
 it follows from  \cite{MP} that the largest graph $\Gamma^-$
appearing in (\ref{lower2}) with nonzero coefficient is the graph
$\Gamma(\emptyset|\mu)$ itself.

Finally we look at possibly disconnected admissible graphs.  Notice
  that the invariant of the disjoint union of two graphs is the
product of invariants.
 We have also remarked that if
$\Gamma_1$ is bigger than $\Gamma_1'$ and $\Gamma_2$ is bigger than
$\Gamma_2'$, then the union of $\Gamma_1$ and $\Gamma_2$ is bigger
than the union of $\Gamma_1'$ and $\Gamma_2'$. Therefore we still
have the leading term being the given graph.

Thus the
correspondence is lower triangular with nonzero diagonal entries.
Such a correspondence is actually invertible as
$I_0$ is lower bounded by Lemma \ref{bounded}.
\end{proof}

\begin{remark} When $\sum_j \mu_j< D\cdot A$, we have
 $l(\eta)-l(\mu)>0$. the largest $\eta$ is $\mu$ followed by
 $D\cdot A -\sum_j \mu_j$ pairs of $(1,
D)$. In the extreme case all $\mu_j=0$, the largest invariant has
$\eta$ with $A\cdot E$ pairs of $(1, D)$. Notice that when $\sum_j
\mu_j< D\cdot A$, then the relative invariant $ \langle [pt],
\varpi|\mu\rangle_{g,A}^{ \tilde X, D}$ is zero by definition. What
Theorem \ref{generalization} says in this case is that $ \langle
[pt], \varpi, \tilde \mu\rangle_{g,p_*(A)}^{ X}$ is expressed as the
sum of standard relative invariants whose weighted graph is lower
than $\langle [pt], \varpi|\mu\rangle_{g,A}^{ \tilde X, D}$.
\end{remark}

\subsection{Correspondence with a point $D$-insertion}
With the application to uniruledness in the mind, we require
 graphs
 have a point $D$-insertion. This is different from \cite{HLR} where the
 point insertion is always an $X$-insertion.

In this subsection we still use $P$ to denote ${\mathbb P}(N_D\oplus
{\mathbb C})$.
\subsubsection{Statement}

 Recall $\iota: D\rightarrow X$ is the
embedding.
 Let
    $$V =min\{ 0<\omega_D\cdot A| A\in H_2(D), <\iota^*\varpi, \tau_{i_1}([pt]), \cdots,
    \tau_{i_k}(\beta_k)>^D_A\neq 0\}.$$
    Here, $\iota^*\varpi$ is of the form $\{
\iota^*\alpha_1, \cdots, \iota^*\alpha_l\}$ with $\deg\alpha_j\ne
2$.

\begin{remark}\label{V}
 By linearity we can assume that  each $\beta_i$
  is in $\Theta$ and each $\alpha_j$ is in $\Xi$.
  Moreover, according to \cite{HLR} $V$ is
achieved by invariants with no descendants. Finally, $V$ is finite
if and only if $D$ is uniruled.
\end{remark}

 Such an invariant determines a standard colored weighted graph
$\Gamma_0(\varpi|\mu)$ in $I$  with
$$\mu=((i_1+1, [pt]), (i_2+1, \alpha_2), \cdots, (i_k+1, \alpha_k)).$$

\begin{definition}\label{D-pt1} We consider the following subset $I_{D-pt}\subset I$ of
colored standard graphs of $X$,

1. $g(\Gamma)=0$,

2. the class $A$ is nonzero and  $\omega(A)\leq V$,

3. admissible graphs with a $D-$point insertion,

4. sup-admissible graphs of the form $\Gamma_0$ with $A\in
im[\iota_*:H_2(D;\mathbb Z)\to H_2(X;\mathbb Z)]$,
\begin{equation}\label{condition1}
\omega(A)=\omega_D(A)=V, \quad  \sum_{t=1}^k(i_t+1)=D\cdot A+1,
\end{equation}

5. empty graph excluded.

 If such a graph is not strictly sup-admissible  we call it
 a restricted graph.
\end{definition}

Let $\mathbb R^I_{D-pt}$ be the  vector subspace  spanned by the
partially ordered set $I_{D-pt}$ of graphs, and $v_{D-pt}^{abs}$ be
the vector of associated absolute invariants.

Notice that all the associated absolute invariants have a point
insertion (possibly descendent).

We also have  the vector of associated relative invariants
$v_{D-pt}^{rel}$ of the symplectic cut, including all the relative
invariants of $(P, D_{\infty})$ with class $A\in
im[\iota_*:H_2(D;\mathbb Z)\to H_2(X;\mathbb Z)]$ satisfying
$\omega(A)=V$, and  insertions of the form $(\tilde \nu|\emptyset)$.

\begin{theorem}\label{minimal}
$A$ restricts to an invertible lower triangular linear
transformation
$$T: {\mathbb
R}^I_{D-pt}\rightarrow {\mathbb R}^I_{D-pt}$$  such that
$$T(v_{D-pt}^{rel})=v_{D-pt}^{abs}.$$
Moreover, there is also an $I_{D-pt, 0}$ version.
\end{theorem}

In the remaining we provide the proof using the following vanishing
results on the relative invariants of $(P, D_{\infty})$.

 \begin{theorem} \label{v}
     Suppose $A$ is  a non-fiber class, i.e.
     $0<\omega_{{D_0}}(\pi_*(A))$.

     (i) If
 $\omega_{{D_0}}(\pi_*(A))<
     V$,
then,
    $$<\varpi,\tau_{i_1}([pt]), \tau_{i_2}(\beta_2[D_0]), \dots,
    \tau_{i_k}(\beta_k[D_0])|\mu>^{P, D_{\infty}}_A=0.$$

(ii) If $\omega_{D_0}(\pi_*(A))=
    V$  and  $$m=\sum_t (i_t+1)\leq D_{0}\cdot A$$ is admissible or
    sub-admissible, then,
    $$<\varpi,\tau_{i_1}([pt], \tau_{i_2}(\beta_2 [D_0]),
    \cdots, \tau_{i_k}(\beta_k [D_0]) |\emptyset>^{P, D_{\infty}}_A=0.$$
    \end{theorem}

Theorem \ref{v} will be proved in the next section.

\subsubsection{Proof of Theorem \ref{minimal}}
As in the proof of Theorem \ref{generalization} we only prove the
version with $\varpi=\emptyset$, i.e. the $I_{D-pt, 0}$ version.

We can assume that the graph $\Gamma(\emptyset|\mu)$ in $I_{D-pt,0}$
is connected.

\noindent{\bf Case I}.
 Let us first look at the
case that $\Gamma(\emptyset|\mu)$ is restricted, or in other words,
admissible.

 Apply the degeneration formula
to it as in the proof of Theorem \ref{generalization}. Notice that
in this case $\delta_{K_1}=[pt]$.

\begin{definition}
 The $P-$graphs $\Gamma^+(\cdots)$ in (\ref{lower2}) are divided into 3 types.

\begin{itemize}

\item  (i) The special graph $\Gamma^+(\tilde \mu|\emptyset)$. In this case the entire curve lies on the
$\mathbb P^1-$bundle side. \footnote{In this case it is tempting to
think that the relative invariant $\langle [pt], \alpha_2,\cdots,
\alpha_k\mid \emptyset \rangle_A^{P, D_{\infty}}$ is the same as the
absolute invariant $\langle [pt], \alpha_2, \cdots, \alpha_k
\rangle_A^{P_D}$. But in general this is not true.}

\item (ii) The connected component of $\Gamma^+$ containing the point insertion is not a fiber curve.

\item (iii) The connected component of $\Gamma^+$ containing the point insertion is a fiber curve
(possibly multiply covered).

\end{itemize}

 The type of  a $\Gamma^-(\emptyset|\eta)$ graph in (\ref{lower2}) is  the type of the company $P-$graph.
In particular, the type (i) $\Gamma^-(\emptyset|\eta)$ graph is just
the empty graph.
\end{definition}

Now let us fix a term in (\ref{lower2}).

 {\bf   Neither $\Gamma^+$ nor $\Gamma^-$ is the empty
graph}.
 Then there are
associated
 classes $B^+\in H_2(P;\mathbb Z)$ and
$B^-\in H_2(\tilde X;\mathbb Z)$ respectively with
$$B^+
+B^-=p_*(B).$$ It follows from (\ref{cohomology relation}) and
Definition \ref{D-pt1},
\begin{equation}\label{area1}
\omega_{\tilde X}(B^-)+\omega_P(B^+)=\omega(B)\leq V.
\end{equation}

Since $B$ is not the zero class, and the graph
$\Gamma(\emptyset|\mu)$ is connected, we have $B^+\ne 0$ and $B^-\ne
0$. Hence it follows from (\ref{area1})
\begin{equation}
\label{area3} 0<\omega_P(B^+)<V, \quad 0<\omega_{\tilde X}(B^-)<V.
\end{equation}

We will show in this case

\begin{prop} \label{restricted}  If the $\Gamma^--$graph
contributes then it is  restricted.

\end{prop}

\begin{proof}
This will be proved by a series of lemmas.

\begin{lemma}\label{D}
For a type (iii) $P-$graph  each relative insertion on the fiber
curve containing the absolute point insertion must  be of the form
$(s, [D])$. Consequently, for each type (iii)
$\Gamma^-(\emptyset\mid \eta)$ graph  there is a point relative
insertion, i.e. $\eta=\{(s, [pt]), \cdots\}$.
\end{lemma}
\begin{proof}
Otherwise the fiber curve cannot meet both the point and the
relative cycle.
\end{proof}

\begin{lemma} \label{area2}
Let $\pi:P\to D_{0}$ be the projection. Then
$$B^+=\pi_*B+(B\cdot
D_{\infty})F.$$
 In particular,
  $$\omega_{D_0}(\pi_*(B^+))\leq \omega_P(B^+)$$ if $B^+\cdot D_{\infty}\geq
  0$,
  and the inequality is strict if $B^+\cdot D_{\infty}$ is positive.
\end{lemma}

\begin{proof}
$H_2(P;\mathbb Z)$ is generated by $H_2(D_{0};\mathbb Z)$ and $F$,
so we can write
$$B^+=B_{0}+mF$$ for some class $B_{0}$ of $D_{0}$.
Since $\pi_*B_{0}=B_{0}$ we have
$$\pi_*B^+=\pi_*B_{0}+0=B_{0}.$$
 On the other hand, since
$B_{0}\cdot D_{\infty}=0$, we have $m=B^+\cdot D_{\infty}\geq 0$.

Since $\omega_P(F)>0$ and $ \omega_{D_0}=\omega_D$, if $m=B^+\cdot
D_{\infty}\geq 0$, then
 $$\omega_P(B^+)=\omega_P(\pi_*B^+)+\omega_P(mF)\geq
\omega_P(\pi_*B^+)=\omega_{D_{0}}(\pi_*B^+)=\omega_D(\pi_*B^+).$$
\end{proof}

\begin{lemma}\label{type2} Type (ii) $P-$graph invariants vanish.
Hence there are no contributing type (ii) $\Gamma^-(\emptyset|\eta)$
graphs in (\ref{lower2}).
\end{lemma}
\begin{proof} Observe that in this case the class $B^+$ is not a fiber
class. And by Lemma \ref{area2} and (\ref{area3}) we have
$$\omega_D(\pi_*(B^+))\leq \omega_P(B^+)< V.$$
The conclusion then follows from part 1 of Theorem \ref{v}.
\end{proof}

Now it follows Lemmas \ref{D}, \ref{area2}, \ref{type2} that the
contributing $\Gamma^-(\emptyset|\eta)$ graphs in (\ref{lower2}) are
still restricted.
\end{proof}

{\bf  $\Gamma^+$ is empty}. In this case $\Gamma^-$ is simply the
given graph $\Gamma(\emptyset|\mu)$.

{\bf  $\Gamma^-$ is empty}. In this case we have

\begin{lemma} \label{specialgraph} Suppose the given graph $\Gamma(\emptyset|\mu)$ is admissible.
Either  the type (i) $P-$graph $\Gamma^+(\tilde \mu|\emptyset)$ is
not allowed or its invariant vanishes. Hence the empty
$\Gamma^-$graph is not a contributing graph in (\ref{lower2}).
\end{lemma}

\begin{proof} The type (i) $P-$graph $\Gamma(\emptyset|\mu)$
does not appear if the class $B$ is not in the image $\iota_*$.

Suppose $B=B^+$ is in the image of $\iota$ and we denote the class
of $D$ still by $B$.
 By our assumption,
 either $\omega(B)=\omega_D(B)<V$ or $\omega(B)=\omega_D(B)=V$.
 Since $B^+$ is not a fiber class,
in either case the vanishing is given by Theorem \ref{v}.
\end{proof}

\noindent{\bf Case II}. Now  consider the case that
$\Gamma(\emptyset|\mu)$ is of the form $\Gamma_0$ satisfying
(\ref{condition1}).

We again apply the degeneration formula to the corresponding
invariant of $X$ distributing all the insertions to the $\mathbb
P^1-$bundle side. Then as argued in Theorem \ref{generalization} the
leading term in (\ref{lower2}) is the special graph.

We only need to  show that the remaining $\Gamma^--$graphs in
(\ref{lower2}) are restricted.

 By Lemma \ref{specialgraph} the type (i) $\Gamma^--$graph does not appear.
 Since $A$ is a minimal uniruled class of $D$, by Theorem \ref{v}, type (ii) $\Gamma^--$graphs do not
appear either. In particular, $\breve \eta$ is constrained by (iii).
For the remaining  $\Gamma^--$graphs, the corresponding $P-$graphs
are of type (iii),  hence  they are restricted.

Thus we have completed the proof of Theorem \ref{minimal}.
\begin{remark}
In fact we have shown there is a sub-correspondence for the admissible graphs in $I_{D-pt, 0}$.
\end{remark}

\section{Relative Gromov-Witten invariants of $(P_D, D_{\infty})$}

     Relative Gromov-Witten invariants of  $\mathbb P^1$-bundle have been studied
     in \cite{MP} and \cite{Ga}.
     We will use
    both their techniques and results extensively.
 In [MP] it is shown using virtual localization that  relative invariants of
a projective $\mathbb P^1-$bundle are determined by absolute
invariants of the base projective manifold.
 We need to apply the symplectic relative virtual localization theorem of
 Chen-Li \cite{CL} in our more general setting.
However, for $\mathbb P^1$-bundle, the theorem is the same as the
corresponding algebro-geometric case in \cite{GV}. We conveniently
use the notations from the algebro-geometric case. Compared to
\cite{MP}, a new ingredient of our case is the point insertion for
which we have to keep track of it at each induction step.

    Let $L$ be a complex line bundle over $D$ and $P_D=\mathbb P(L\oplus \mathbb C)$.
    Then $\pi: P_D\rightarrow D$  is a $\mathbb P^1$-bundle with zero section $D_0$ and
    infinity section $D_{\infty}$. Clearly, $D_0, D_{\infty}$ are isomorphic to $D$.
    In this section, we calculate certain relative Gromov-Witten
    invariants of $P_D$ relative to the infinity section
    $D_{\infty}$. First of all, there are two kinds of cohomology
    classes: $\pi^*\alpha$ and $\beta[D_0]$. As previously remarked, the class $\beta[D_0]$, which is the cup product of the pull-back
    of $\beta$ with the Poincar\'e dual of $D_0$, corresponds to the
    Poincar\'e dual of $\beta$ viewed as a homology cycle of
    $D_0$.

    Let $A\in H_2(D_{\infty};\mathbb Z)$. We view it as a homology class
    of $P_D$.
    We shall consider relative invariants
    of the form
    \begin{equation}\label{relative}
    <\varpi, \tau_{i_1}([pt][D_0]), \tau_{i_2}(\beta_2[D_0]), \cdots,
    \tau_{i_k}(\beta_k[D_0])|\mu>^{P_D,
    D_{\infty}}_{0,A},
    \end{equation}
    where $\varpi$ consists of insertions of the form $\pi^*(\alpha_1), \cdots,
    \pi^*(\alpha_l)$.

    To evaluate such an invariant we will need to study other types
    of relative invariants, twisted rubber invariants, as well as twisted invariants of the divisor.
We will write $P$ for $P_D$.

\subsection{Twisted invariants of $D$}\label{twist-D}
There is an important twisted Gromov-Witten theory treated by Farber
and Pandharipande in \cite{FP} and Coates and Givental \cite{CG}.

    Suppose that $f: \Sigma\rightarrow D_{0}$ is a stable map.
    For the  line bundle $L\to D_{0}$ we can define a virtual bundle
\begin{equation}\label{virtual bundle}H^1(f^*L)-H^0(f^*L)
\end{equation}
     over the
    moduli space $\overline{\M}^{D_0}_{0,k+l}(A)$ of stable maps along with its Euler class $e$. For
    example, we use the $S^1$-action on the fiber to define
     the  equivariant Euler class of (\ref{virtual bundle}), still denoted by $e$,
 and then take non-equivariant limit.

Associated to the relative invariant (\ref{relative}) we have the
Gromov-Witten
     invariant of $D_{0}=D$ twisted by $e$   defined as
    \begin{equation}
    \begin{array}{ll}
    &<\iota^*\varpi, \tau_{i_1}([pt]), \tau_{i_2}(\beta_2), \cdots,
    \tau_{i_k}(\beta_k)>^{
    D,
    e}_{0,A}\cr
    =&\int_{[\overline{\M}^{D}_{0,k+l}(A)]^{vir}}\prod_s
    \iota^*\alpha_s  [pt]\psi_1^{i_1}\prod_{t=2}\beta_t\psi_t^{i_t} \wedge e.
    \end{array}
    \end{equation}
    The above twisted invariant has been studied in  \cite{FP} and \cite{CG}.
    Their idea is  to  mark an additional point.
We summarize in the following form.

     \begin{prop}\label{diagonal proposition}
    A  twisted invariant can be expressed in terms of a  similar invariant replacing $e$ by a descendent at the
    additional marked point together with some products of ordinary invariants.
     The total curve class of each product of invariants is still equal to $A$,
     each insertion of the original invariant appears as an
     insertion of a factor invariant,
     and the factor invariants are
     linked by dual insertions from the diagonal class  .
     \end{prop}

These products have the stated properties because they
     correspond to
     boundary contributions from nodal curves and are obtained via diagonal
     splitting.

     In our situation, we also need to consider a slight variant of the above
     twisted invariants. Here, we consider the following complex
    \begin{equation}\label{complex} D_{f,L}: \{, v\in \Omega^0(f^*L), v(x_i)=0, i\leq
    k\}\rightarrow \Omega^{0,1}(f^*L).
    \end{equation}
    $\hbox{coker} D_{f,L}-\ker D_{f,L}$ defines a virtual bundle of rank
    \begin{equation}
    k-c_1(L)(A)-1.
    \end{equation}
     Let $e_L$ be its Euler class. Then, the Gromov-Witten
     invariant twisted by $L$ is     defined as
    \begin{equation}\label{g twist}
    \begin{array}{ll}
    &<\iota^*\varpi, \tau_{i_1}([pt]), \tau_{i_2}(\beta_2), \cdots,
    \tau_{i_k}(\beta_k)>^{
    D,
    L}_{0,A}\cr
    =&\int_{[\overline{\M}^D_{0,k+l}(A)]^{vir}}\prod_s
    \iota^*\alpha_s [pt]\psi_1^{i_1}\prod_{t=2}\beta_t\psi_t^{i_t} \wedge e_L.
    \end{array}
    \end{equation}
    We call it a {\em generalized $L$-twisted invariant}.
    There is a short exact sequence
    $$0\rightarrow \ker D_{f,L}\rightarrow H^0(f^*L)\rightarrow \oplus_x L_{f(x_i)}\rightarrow \hbox{coker} D_{f,L}
    \rightarrow H^1(f^*L)\rightarrow 0.$$
    As a virtual bundle,
\begin{equation} \label {exact sequence}\hbox{coker} D_{f,L}-\ker
    D_{f,L}=H^1(f^*L)-H^0(f^*L)-\oplus_i L_{f(x_i)}.
    \end{equation}
    Observe that we can use (\ref{exact sequence}) to express $e_L$ in terms of the Euler class  $e$ of
   the  virtual bundle (\ref{virtual bundle})
     and the insertion $c_1(L)$ .  It is then
not hard to see that  Proposition 4.1 applies to generalized
$L$-twisted invariants of the form (\ref{g twist})  as well.

\subsection{Type I/II relative invariants and twisted rubber invariants}

Let $Y$, $\mathcal L$ and $\mathcal R$ all denote the $\mathbb P^1$
bundle $\mathbb P(N_D\oplus \underline {\mathbb C})$.

 For any non-negative integer $m$, construct $\mathcal R_m$ by
gluing together $m$ copies of $\mathcal R$, where the infinity
section of the $i^{th}$ component is glued to the zero section of
the $(i+1)^{th}$ component for $1\leq i \leq m$. Denote the zero
section of the $i^{th}$ component by $D_{i,0}$, and the infinity
section by $D_{i, \infty}$, so  $Sing \mathcal R_m =
\cup_{i=1}^{m-1} D_{i, \infty}$.  Define $Y_m$ by gluing $ Y $ along
its infinity section denoted by $D_{0,\infty}$ to $\mathcal R_m$
along $D_{1,0}$. Thus $Sing Y_m = \cup_{i=0}^{m-1}D_{i, \infty}$.
$Y_0=Y$ will be referred to as the level $0$ component and the
$\mathcal R_i$ will be called the level $i$ component. We will also
sometimes denote $D_{m, \infty}$ by $D_\infty$ if there is no
confusion.

Let $\mbox{Aut}_D \mathcal R_m$ be the group of automorphisms of
$Q_m$ preserving each $D_{i, 0}, D_{i, \infty}, 1\leq i\leq m$, and
the morphism to $D_{1,0}$. And let $\mbox{Aut}_D Y_m$ be the group
of automorphisms of $Y_m$  with restriction to $\mathcal R_m$ being
contained in $\mbox{Aut}_D \mathcal R_m$. Clearly, $\mbox{Aut}_D Y_m
= \mbox{Aut}_D \mathcal R_m\cong ({\mathbb C}^*)^m$, where each
factor of $({\mathbb C}^*)^m$ dilates the fibers of the ${\mathbb
P}^1-$bundle $\mathcal R_i\longrightarrow D_{i,0}$. Denote by
$\pi[m] : Y_m\longrightarrow Y$ the map which is the identity on the
root component $Y_0$ and contracts all the bubble components to
$D_{1,0}$ via the  fiber bundle projections.

Similarly we can form $\mathcal L_n$ and glue it to the left of
$Y_m$. We denote the resulting chain of $\mathbb P^1-$bundle by
$_nY_m$. If $m=0$  we simply write $_nY$ for $_nY_0$. Of course
$Y_m$ is the same as $_0Y_m$. $\mathcal L_j$ is regarded as the
level $-j$ component. The automorphism group for $\mathcal L_n$ is
defined in the same way as for $Y_m$, and the extension to $_nY_m$
is obvious.

\subsubsection{Type I invariants and Type II invariants }
Given a relative insertion $\mu$ at $D_{\infty}$, the relative
moduli space for $(Y, D_{\infty})$ consists of the union over $m$ of
equivalence classes of marked relative stable maps into $(Y_m, D_{m,
\infty})$
 satisfying the relative constraint. It comes with a virtual
 fundamental homology class.

 There are also
natural cohomology classes associated to marked points:  the pull
back classes via the evaluation maps at the marked points, as well
as descendent classes. Thus we can integrate these classes over the
virtual class to define relative invariants, called type I
invariants in \cite{MP}.

In the same way we can define  relative invariants for $(Y, D_0)$ by
considering marked relative stable maps  into $(_nY, D_{-n, 0})$ for
various. These invariants are also called type I invariants.

Type II invariants are relative invariants for $Y$ relative to both
$D_0$ and $D_{\infty}$. They are defined via  moduli spaces of
equivalence classes of  marked stable maps into $(_nY_m, D_{-n, 0},
D_{m,\infty})$.

\subsubsection{Distinguished type II invariants and their orders}
A distinguished invariant of type II is an invariant of $(Y,  D_0
\cup D_{\infty})$ with a distinguished insertion of the form
$[D_{\infty}]\delta$ (Notice that our definition is different from
that in \cite{MP} where $[D_0]$ is in place of $[D_{\infty}]$).

Distinguished type II invariants are ordered in a similar way as in
Definition \ref{order}. The new features are  that parts  (3) and
(4) of that order are replaced successively  by the number of
non-distinguished insertions, the total degree of $D_0-$relative
insertions, the total degree of $D_{\infty}-$relative insertions,
the degree of the distinguished insertion.

\subsubsection{Non-rigid targets and rubber invariants with $\Psi$ insertions}
\label{twisted rubber}

Let we denote by ${_nY_m}^{\sim}$ the collapsing of  $Y_0$ in
$_nY_m$. ${_nY_m}^{\sim}$ is called a non-rigid target, due to the
$\mathbb C^*$ action at each level.

Given relative insertions $\mu$ and $\nu$ on the two sides $D_{-n,
0}$ and $D_{m,\infty}$, there is the rubber moduli space
$\overline{\M}^P_{0,l}(B;\mu|\nu)^{\sim}$ consisting of the union
over $m, n$ of equivalence classes of stable pseudo-holomorphic maps
from $l-$marked genus 0 curve into ${_nY_m}^{\sim} $  with class
$B$. Notice that here the stability is the rubber stability and the
equivalence is the rubber equivalence.

As for the ordinary moduli space, a rubber moduli space  also comes
with the virtual fundamental (homology) class. Evaluation classes
and descendent classes are invariant under enlarged rubber
automorphisms and hence define classes on the rubber moduli space.
Thus we can integrate these classes over the virtual class to define
invariants, which are called rubber invariants in \cite{MP}.

In fact,  there are two additional  degree 2 cohomology classes on
the rubber moduli space coming from two classes on the relative
moduli space, $\Psi_0$ and $\Psi_0$. They are the cotangent classes
at $D_0$ and $D_{\infty}$ respectively. We give the description of
$\Psi_0$ here following \cite{Ga}, the construction for
$\Psi_{\infty}$ is the same. Given a relative map $f$ with a marked
point $x_i$ mapping to $D_{0}$, if the multiplicity of $f$ at $x_i$
is $\alpha_i$, then define
\begin{equation}\label{Psi}
\Psi_0=\alpha_i\psi_i+ev_{x_i}^*c_1(N_{D_{0}|Y})=\alpha_i\psi_i+ev_{x_i}^*c_1(L).
\end{equation}
It is independent of the choice of $x_i$ as along as it is mapped to
$D_{\infty}$.

 As  $\Psi_0$ and $\Psi_{\infty}$ are invariant under enlarged rubber
automorphisms they descend to degree 2 classes, still called
$\Psi_0$ and $\Psi_0$, on rubber moduli spaces. An important new
feature is that these classes are not generated by evaluation
classes. We call rubber invariants with $\Psi$ insertions {\it
twisted} rubber invariants. We will see very soon that twisted
rubber invariants appear in the relative virtual localization
formula. They are generally very difficult to evaluate  explicitly.

We can certainly define  type II invariants with $\Psi$ insertions.
 We can also define type I invariants for $(Y, D_{\infty})$ with
$\Psi_0$ insertions, and type I invariants for $(Y, D_{0})$ with
$\Psi_{0}$ insertions. However, we will have no use for these more
general invariants. Thus we reserve the name a type I or type II
invariant only for one with no $\Psi$ insertions.

\subsection{Relative virtual localization on $\mathbb P^1$ bundle}
In this subsection we follow p. 135-8 in \cite{Ga}.

We switch back to the notation $P$, i.e. $P=Y$ and $_nP_m=_nY_m$.
     $P$ carries a
    natural $S^1$-action by rotating the fibers. The fixed point
    loci
    are precisely $D_0$ and $ D_{\infty}$. It induces a natural action on
        the moduli space of relative stable maps
    $\overline{\M}^{P, D_{\infty}}_{0, k+l}(A)$.

\subsubsection{Fixed point components labeled by bipartite graphs}

    The set of connected components of fixed point locus  is indexed by bi-partite
    graphs.
    Each vertex corresponds  to  either the ordinary (connected) stable maps into $D_0$ or
    stable maps into rubber over $D_{\infty}$. The first type of
    vertex is called an ordinary vertex, and the second type is
    called a rubber vertex.

    Each edge corresponds to
    Galois covers of fibers of $P$  totally ramified over $D_{0,0}$
    and $D_{0, \infty}$ with no marked points away from $D_{0,0}$
    and $D_{0, \infty}$.

The connection data of the Galois covers is described by a sum over
cohomology weighted partitions specifying the rubber relative
conditions on the connecting divisor.

    Given each bipartite graph, each component $F$ is then a finite quotient of a
    fiber product  $M_0\boxtimes \mathcal M_1$ by a finite group $G$. Here

\begin{itemize}
    \item $M_0$ is the product over vertices on $D_0$
    of ordinary moduli spaces $\overline{\M}^{D}_{0,l}(B)$ to $D_0$,
    adding the nodes $y_i$ where the fibers are attached;

    \item $\mathcal M_1$ is the product  of the rubber moduli space
    $\overline{\M}^P_{0,l}(B;\mu|\nu)^{\cong}$
     over vertex on $D_{\infty}$, adding the nodes $y_i$ where the fibers are attached;

\item $\boxtimes$ denotes a fiber product over evaluation maps to $D_{0,\infty}=D_{1,0}$
at the gluing points $y_i$ (in the intersection of levels $0$ and
$1$),

\item $G$ is the group of permutations of the points $y_i$ that
preserves the multiplicities.

\end{itemize}

The virtual fundamental class of $F$ is the one induced by this
product structure.

\subsubsection{Equivariant Euler class}
 For each component
$F\subset \overline{\M}^{P, D_{\infty}}_{0, k+1}(A)$, the
equivariant Euler class of the virtual normal bundle
$N^{virt}_{F/\overline{\M}^{P, D_{\infty}}_{0, k+1}(A)}$ is also a
fiber product,
$$e(N^{virt}_{F/\overline{\M}^{P, D_{\infty}}_{0, k+1}(A)})=
\frac{1}{G}(\frac{[M_0]^{virt}}{e_0(N^{virt}_{F/\overline{\M}^{P,
D_{\infty}}_{0, k+1}(A)})} \boxtimes \frac{[\mathcal
M_1]^{virt}}{e_1(N^{virt}_{F/\overline{\M}^{P, D_{\infty}}_{0,
k+1}(A)})}).
$$

$e_0$ is a product with one factor being the equivariant Euler class
of the virtual bundle
\begin{equation}
H^0f^*N_{D_0/P})- H^1(f^*N_{D_0/P})=H^0(f^*L)-H^1(f^*L)
\end{equation}
 on
$M_0$. This is why we need to consider the twisted invariants as in
\ref{twist-D}.

 There are two other types of factors of $e_0$. Denote the generator of
$H^*_{\mathbb C^*}(pt)$ by $t$. Every node $y_i$ in $D_{0,0}$
connecting to a multiple cover of degree $m$ contributes a product
$$
(\frac{t+ev_{y_i}^*c_1(L)}{m}-\psi_{y_i})\prod_{k=1}^{m-1}\frac{k(t+ev_{y_i}^*c_1(L))}{m},
.$$ Here the first term corresponds to part (iii) on p. 137 in
\cite{Ga}, and the second term corresponds to (ii)(b) there.


Concerning $e_1$, each marked point $x_i$ in $D_0$ on a multiple
cover of degree $m$ contributes to $e_1$ a product
$$
(-t-\Psi_0)\prod_{k=1}^m\frac{k(t+ev_{x_i}^*c_1(L))}{m}.
$$
Here the first term corresponds to part (iv) on p. 137 in \cite{Ga},
and the second term corresponds to (ii)(a) there. This is why we
need to consider twisted rubber invariants as in \ref{twisted
rubber}.

\subsubsection{Virtual localization formula}
The virtual localization formula equates
$$[\overline{\M}^{P, D_{\infty}}_{0,
k+1}(A)]^{virt}=\sum_F\frac{[F]^{virt}}
{e(N^{virt}_{F/\overline{\M}^{P, D_{\infty}}_{0, k+1}(A)})}$$ in the
equivariant cohomology of $\overline{\M}^{P, D_{\infty}}_{0,
k+1}(A)$.

Thus, as mentioned,  to apply the virtual relative localization we
need to evaluate an ordinary twisted invariants over $M_0$ and
certain twisted rubber invariants. As mentioned in 4.1, ordinary
twisted invariants have been  treated in \cite{FP}.  As for twisted
rubber invariants, a nice algorithm, the rubber calculus, has been
developed in \cite{MP}.

\subsection{Rubber calculus}

In this subsection we review the rubber calculus  and in addition,
we keep track how the curve class behaves.

\subsubsection{The first reduction--removing $\Psi_0$} The
first step is to remove $\Psi$ insertions from  rubber invariants,
i.e. express twisted rubber invariants in terms of (ordinary) rubber
invariants.

It involves the dilaton equation, the divisor equation and the
topological recursion relation. We only describe them in the form
needed.

{\bf Dilation equation}: if $c=2g-2+n+l(\mu)+l(\nu)\ne 0$ for a
rubber invariant, then it  is $c^{-1}\times$   the rubber invariant
with an extra absolute insertion $\tau_1(1)$. The curve class is
again preserved.

{\bf Divisor equation}: if a divisor is nonzero on the curve class,
a $\Psi_0^k-$rubber invariant is the sum of the rubber invariant
with the divisor added as  an absolute insertion,  rubber invariants
with smaller descendent powers, and $\Psi_0^{k-1}-$rubber
invariants. The curve classes are still preserved.

 Finally we describe the application the topological recursion to a
 rubber invariant with a $\Psi_0$ insertion and {\it at least one absolute insertion}. The starting
point is that, by (\ref{Psi}),  $\Psi_0$ over the rubber moduli
space can be expressed as $ev^*_p(c_1(L))+\alpha^*\psi_3$, where $p$
is the marked point carrying one absolute  insertion. $\alpha$ is
the canonical map to the Artin stack of $3-$pointed genus 0 curves,
sending a rubber map $f$ to the fiber $C_f$ containing $p$ with
marked points $D_0\cap C_f, f(p), D_0\cap C_f$.

{\bf Topological recursion}: $\psi_3$ is expressed as a sum of
boundary divisor, hence a  $\Psi_0^k-$rubber invariant with at least
one absolute insertion is the sum of a $\Psi_0^{k-1}-$invariant, and
 $\Psi_0^{k-1}-$rubber invariants multiplied by  a  rubber invariant
 with no $\Psi_0$ insertion and at least one absolute insertion. The curve classes are
 possibly smaller.

\begin{remark} To be able to apply  topological
recursion, we need at least one absolute insertion. This can be
achieved either by the dilaton equation or the divisor equation.
\end {remark}

For a fiber class $\Psi_0^k-$rubber invariant, the target stability
insures that the relevant $c$ in the dilaton equation is nonzero,
hence   it can be turned into a $\Psi_0^k-$rubber invariant with an
extra dilaton insertion, and hence at least one absolute insertion.
 Apply the topological recursion to reduce the
dilaton rubber invariant to rubber invariants  with fewer $\Psi_0$
insertions. Repeating the cycle yields rubber invariants without
$\Psi_0$ insertions.

For a non-fiber class $\Psi_0^k-$rubber invariant,  first add a
divisor insertion $\pi^*\omega_D$. As $\pi^*\omega_D$ pairs
positively with a non-fiber class, the divisor equation can be used
to express the original invariant in terms of the new invariant
which has the divisor insertion and hence at least one absolute
insertion, together with either $\Psi_0^{k-1}-$rubber invariants or
$\Psi_0^k-$rubber invariants with smaller descendent powers.

Notice that invariants with a negative power descendent insertion is
automatically zero. Thus we can  employ the topological recursion
repeatedly until there are no $\Psi_0$ insertions.

 \subsubsection{The second reduction--Rubber to type II}

This is achieved through the rigidification process.
  A rubber invariant with {\it at least one
absolute insertion} is turned into a type II invariant with one
absolute insertion replaced by its product with the  divisor
$D_{\infty}$. The curve class is preserved.

A fiber class rubber integral with no $\Psi_0$ insertions is turned
into  type II invariants through rigidification after adding a
dilaton insertion.

While a non-fiber class rubber integral with no $\Psi_0$ insertions
is expressed in terms of type II invariants through rigidification
after adding a divisor insertion of the form $\pi^*\omega_D$.

In summary, the following is proved in \cite{MP}.

\begin{prop} \label{rubber}Any $\Psi^k_{\infty}-$rubber invariant can be expressed
in terms of  type II invariants with the same class, or products of
type II invariants. The total curve class of each product of
invariants is still equal to $A$, each insertion of the original
invariant appears as an
     insertion of a factor invariant, and the adjacent factor invariants
are linked by dual insertions.

Moreover, if the curve class is non-fiber and the domain is
connected, the type II invariants with the same curve class are
distinguished and their relative insertions are no bigger than those
of the rubber invariant, the number of non-distinguished insertions
is non bigger than the number of absolute insertion of the rubber
invariant.

\end{prop}

\subsection{The reduction algorithm----relative to divisor}

Let us now finally describe the algorithm in \cite{MP} to determine
relative invariants of $(P, D_{\infty})$ from invariants of $D$.
This step involves both  virtual localization and degeneration.

For a fiber class, the moduli space of relative stable maps fibers
over $D$ with fiber isomorphic to the moduli space of stable maps to
$\mathbb P^1$ relative to $\infty$. Thus a fiber class
 invariant can be expressed in
terms of the classical cohomology of $D$ and equivariant GW
invariant of $(\mathbb P^1,\infty)$.

\subsubsection{Three
relations and the final reduction}

Now consider a class $A$ with $\pi_*(A)\ne 0$, i.e. a non-fiber
class.

 By first applying localization
 a type I relative invariant of $(P, D_{\infty})$
can be expressed in terms of twisted rubber invariants (with
$\Psi_0$ insertions) and twisted integrals in the GW theory of $D$.

Then Proposition \ref{rubber}
  expresses the involved twisted rubber invariants  in terms of
distinguished type II invariants. And Proposition \ref{diagonal
proposition} expresses the involved twisted integrals in the GW
theory of $D$ in terms of ordinary invariants of $D$.

 Thus it suffices to show every distinguished type II
invariants can be expressed in terms of invariants of $D$. This is
achieved through  the following 3 relations.

For relative conditions
$$\mu=\{(\mu_i, \delta_{r_i})\},\nu=\{(\nu_j, \delta_{s_j})\},$$
consider a $(Y, D_{\infty})$ type I invariant of the form
\begin{equation}\label{typeI}
<\prod_i\tau_{\mu_i-1}([D_0]\delta_{r_i})\omega\tau_0([D_\infty]\delta)|\nu>^{P,
D_{\infty}}_A,
\end{equation}
and the distinguished type II invariant
\begin{equation}\label{typeII}
<\mu|\omega\tau_0([D_\infty]\delta)|\nu>_A.
\end{equation}

Relation 1, proved by the degeneration formula for (\ref{typeI})
relative to $D_0$, expresses the distinguished type II invariant
(\ref{typeII}) in terms of the type I invariant (\ref{typeI}),
undistinguished type I invariants of $(P, D_{\infty})$ with class
$A$, type I invariants of $(P, D_{\infty})$ with smaller curve
classes, and distinguished type II invariants lower than
(\ref{typeII}).

Relation 2, proved again by virtual localization, Proposition
\ref{diagonal proposition} and Proposition \ref{rubber}, expresses
the type I invariant (\ref{typeI}) in terms of distinguished type II
invariants with class A but strictly lower than (\ref{typeII}), type
II invariants with class smaller than $A$ and  invariants of $D$. We
briefly review the argument.

In the virtual localization formula, there are contributions from
$M_0$ and $\mathcal M_1$ respectively.
 The contribution from $M_0$
is a twisted integral in the GW theory of $D$, which can be reduced
to ordinary GW invariants of $D$ by Proposition \ref{diagonal
proposition}.

 Thus the  principal terms of the localization formula
come from fixed loci with constant $D_0$ vertices.   Apply
Proposition \ref{rubber} to the rubber invariants, which have  the
curve class $A$ and  a $\Psi_{0}$ insertion, we obtain distinguished
type II invariants with curve class $A$, together with products of
type II invariants with the same total  curve class and including
each original insertion.

Similarly, Relation 2' expresses undistinguished type I invariants
with class $A$ in Relation I  in terms of distinguished type II
invariants with class $A$ which are lower than (\ref{typeII}),
products of type II invariants with the same total  curve class, and
products of GW invariants of $D$ the same total  curve class.

 By  Relations 1, 2, and 2',
a non-fiber class distinguished type II invariant can be inductively
computed from products of lower order distinguished type II
invariants together with invariants of $D$ with the total curve
class. Here the primary induction is on the pair $(g, A)$, and the
secondary induction is on the ordering.

In summary, the following is proved in \cite{MP}.

\begin{prop} \label{typeI/II}Any type I/II invariant can be expressed
in terms of   invariants of $D$ with the same class, or products of
invariants of $D$. The total curve class of each product of
invariants is still equal to $A$, each insertion of the original
invariant appears as an
     insertion of a factor invariant, and the adjacent factor invariants
are linked by dual insertions.

\end{prop}

\subsection{Two vanishing theorems for invariants with a descendent point
insertion}\label{vanishing} Our interest here is to relate relative
invariants with a descendent point insertion to invariants of $D$
with a descendent point insertion.

    \subsubsection{The first vanishing theorem}
    By keeping track of the reduction scheme of \cite{MP} we have the following vanishing theorem

    \begin{theorem} \label{MP}
     Suppose that a non-fiber class $B$ has the property that $\pi_*(B)$ is not a sum of
     a uniruled class and an effective class of $D$. Then,
    $$<\varpi, \tau_{i_1}([pt]), \tau_{i_2}(\beta_2[D_0]), \dots,
    \tau_{i_k}(\beta_k[D_0])|\mu>^{P, D_{\infty}}_B=0.$$
    \end{theorem}

\begin{proof}
 To simplify the computation, we consider the following  invariant subset
    of the moduli space $\overline{\M}^{P, D_{\infty}}_{0,k+l}(A)$ of relative stable maps.
Without loss of generality, we choose a
    point, denoted by $S_1=pt$, and
    choose submanifolds $S_i, 2\leq i\leq k, $ of $D_0$ representing $\beta_1=[pt]$ and $\beta_i[D_0]$. We here assume
     that $S_i, i\geq 1$
    intersect transversely. Let
\begin{equation}\begin{array}{ll}
    &\overline{\M}^{P, D_{\infty}}_{0,k+l}(A, pt, S_2,
    \cdots, S_k)\cr
    =&\{f\in \overline{\M}^{P, D_{\infty}}_{0,
    k+l}(A); f(x_0)=pt, f(x_i)\in S_i, 1\leq i\leq k\}.
    \end{array}
\end{equation}
We can construct a  virtual fundamental cycle for this space
    which will give the desired relative invariants.
     Of course we need to modify the deformation-obstruction complex. The
    linearization of the Cauchy-Riemann operator is the complex
    $$D_f: \{v\in \Omega^0(f^*TP), v(x_i)\in TS_i\}\rightarrow
    \Omega^{0,1}(f^*TP).$$

As each $S_i$
    is in the fixed locus $D_0$,
     $\overline{\M}^{P, D_{\infty}}_{0,k+l}(A, pt, S_2,
    \cdots,
    S_k)$ also carries a natural $S^1$-action.
Thus we still can
    apply the localization formula to this moduli space. The only
    difference here is that we need to calculate the fixing-moving part
    of our new complex to determine the contributions.


  Since $\pi$ is an equivariant map, each $\pi^*(\alpha_i)$ is an equivariant class
    containing no equivariant parameter. Applying the localization
    formula, the relative invariant is given by
    \begin{equation}\label{bipartite graph}
    \sum_F\int_{[\M_F]^{vir}} \frac{\prod_t
    \bar{\psi}^{i_t}_t\varpi}{e(N_{vir})},
    \end{equation}
    where each $F$ is indexed by a bipartite graph and  $\bar{\psi}_i$ is the equivariant extension of $\psi_i$.

     Recall that for each bipartite graph the corresponding component of the fixed point loci is the fiber product of
     a stable map moduli space of $D_0$ and a rubber stable map moduli space.
Each vertex of the bipartite graph either  corresponds to a
connected component of
    stable maps into $D_0$ or a connected component of rubber
    stable maps into $D_0$. We call the vertex {\em constant
    vertex} if its homology class is zero.

Choose a self dual basis $\{\alpha_i\}$ of $H^*(D;\mathbb R)$ and
let $\{\check \alpha_i\}$ be its dual basis. Then the diagonal class
of $D$ is simply
     \begin{equation}\label{diagonal}[\Delta]=\sum_i\alpha_i\otimes \check \alpha_i.
     \end{equation}

Fix a bipartite graph $\Lambda$ corresponding to a component $F$.
Let $E$ be the set of oriented edges of $\Upsilon$.
 For each edge of
$\Lambda$, we
     insert the diagonal class $[\Delta]$ in (\ref{diagonal}) to split $[\mathcal M_F]^{vir}$  as
     a disjoint union of
      products of virtual fundamental cycles, where each product has factors indexed by vertices of $\Upsilon$.
The union is over the space $T$ of the maps $E$ to the set of
ordered pairs $\{(\alpha_i,\check \alpha_i)\}$ such that two edges
with opposite orientations map to opposite pairs. For each $P\in T$
we define $P_{\lambda}$ to be the set of elements of $\{\alpha_i\}$,
each coming from $P$ and an edge originating at $\lambda$.

We thus can   write  explicitly the contribution of each bipartite
graph as a sum
    \begin{equation} \label{edge} \sum_{P\in T}\prod_{vertices} V_{\lambda, P|_{\lambda}},
    \end{equation}
    where $V_{\lambda, P|_{\lambda}}$ is
     the twisted or rubber invariant at $\lambda$ with added marked points constrained
     by $P|_{\lambda}$.

    Clearly  in each summand of (\ref{edge}) there is a vertex $\lambda$ such that  $V_{\lambda, P_{\lambda}}$
    contains a descendent point insertion.

    For each vertex $\lambda$, we denote its curve class by
    $B(\lambda)$. It is clear that
    $\pi_*(B)=\sum_{\lambda}\pi_*(B(\lambda)).$
    For each $D_0$ vertex, $V_{\lambda}$ is a {\it generalized} twisted Gromov-Witten invariant by normal bundle of $D_0$,
    of the form (\ref{g twist}).
    For each rubber vertex
    $\lambda$ with class $B(\lambda)$,
    $V_{\lambda}$ is a twisted rubber invariant.

    By the construction, one of $V_{\lambda}$'s contains a point insertion. The remaining argument is
    to keep track of the point insertion on the invariant with
    nonzero curve class.
The  case where the vertex $\lambda$ containing a point insertion
    could be a constant vertex is dealt with in the next lemma which we will prove afterwards.

   \begin{lemma}\label{constant vertex}
       If the vertex ${\lambda}$ containing a point insertion is a
       constant vertex, then there is another vertex $\lambda'$
        in the graph such that  $B(\lambda')\neq 0$ and
        $V_{\lambda'}$ contains a point insertion.
        \end{lemma}

    \begin{lemma} \label{constant vertex2}
      Each
      $V_{\lambda, P|_{\lambda}}$ can be expressed as
      \begin{equation} \label{nonconstant}
    \sum_{\pi_*(B(\lambda))=B_1+\cdots+B_k}a <,>^D_{B_1}\cdots <,>^D_{B_i}\cdots
    <,>^D_{B_k},
    \end{equation}
    where $a$ is a combinatoric constant.
    \end{lemma}




 Suppose
 $$<\varpi, \tau_{i_1}([pt]), \tau_{i_2}(\beta_2[D_0]), \dots,
    \tau_{i_k}(\beta_k[D_0])|\mu>^{P, D_{\infty}}_B\neq 0.$$
     Then one summand in (\ref{bipartite graph}) is
    nonzero. Hence one summand in the the corresponding (\ref{edge}) is
    nonzero, and
 this possible only if  one $V_{\lambda, P|_{\lambda}}$ is nonzero.

    By Lemmas  \ref{constant vertex} and  \ref{constant vertex2}
     we can express $\pi_*(B(\lambda))$ as
    an effective decomposition such that one of the summands is a uniruled
    class.
    This contradicts to the assumption.
\end{proof}
We now prove Lemma \ref{constant vertex}

\begin{proof}

    We have to show that the point insertion can
    always be applied to a nonconstant vertex. Suppose that the point
    insertion is on a constant vertex in $D_0$. We claim  that this vertex
    cannot contain any other $D$-insertions. Suppose
    that the constant vertex contains  absolute marked point
    $j_1, \cdots, j_h$. There is an additional marked point from the edge connecting it to a rubber vertex.
    The moduli space is then $\overline{\M}_{0,h+1}$.


     Via the diagonal insertion process it is easy to see that the contribution of the constant vertex under
     consideration is the summation of terms
     $$\int_{\overline{\M}_{0,h+1}} \alpha_i H,$$
     for some equivariant class $H$.
     This term is nonzero only if
     $\alpha_i=1$. Hence, its dual $\check \alpha_i=[pt]$ is
     inserted into the contribution of the rubber component.

     If the
     rubber vertex under consideration is also constant. We can
     apply the same argument to transport the point insertion into
     the next vertex along the graph until reaching a nonconstant
     vertex.


    \end{proof}

    Next, we prove Lemma \ref{constant vertex2}.

    \begin{proof}
 There are two types of vertices, $D_0$-vertices and rubber
        vertices.

        If $\lambda$ is a $D_0$-vertex, $V_{\lambda, P|_{\lambda}}$ is a
        generalized $L-$twisted Gromov-Witten invariant.
        We simply apply Farber-Pandharipande's
        argument.

        As mentioned in 4.1, the basic idea is that by marking an additional
        point we can express $V_{\lambda,P|_{\lambda}}$ as an ordinary
        Gromov-Witten invariant of $D$ with an additional marked
        point plus the boundary contributions. By Proposition \ref{diagonal proposition}, the boundary
        contribution can be written as
        \begin{equation} \label{nonconstant}
    \sum_{\pi_*(B(\lambda))=B_1+\cdots+B_k}a <,>^D_{B_1}\cdots <,>^D_{B_i}\cdots
    <,>^D_{B_k},
    \end{equation}



     When $\lambda$ is a rubber vertex, the much more complicated
    argument of Maulik and Pandharipande achieves the same for the twisted rubber invariant
    $V_{\lambda,P|_{\lambda}}$.

    \end{proof}

    Recall
    $$V=min\{\omega(A),<\iota^*\varpi, \tau_{i_1}([pt]), \tau_{i_2}(\beta_2), \dots,\tau_{i_k}(\beta_k)>^D_A\neq 0\}.$$
    The above theorem implies that the first part of Theorem of
    \ref{v}.

    \begin{cor} \label{MP1}
     Suppose that a non-fiber class $B$ has the property that $\omega(B)<V$. Then,
    $$<\varpi, \tau_{i_1}([pt]), \tau_{i_2}(\beta_2[D_0]), \dots,
    \tau_{i_k}(\beta_k[D_0])|\mu>^{P, D_{\infty}}_B=0.$$
    \end{cor}

    \begin{proof}
     We observe that $B=B_0+|\mu|F$ for
     $B_0=\pi_*(B)\in H_2(D_0, \Z)$. Therefore,
     $\omega(\pi_*(B))\leq \omega(B)<V$. By the definition  of $V$,
     $\pi_*(B)$ can not be the sum of an uniruled class and an effective class of $D$.
     \end{proof}

\subsubsection{The second vanishing theorem}
    Another case of interest is the special graph in the
    degeneration formula where the degeneration graph lies
    completely in $P$. This is an admissible invariant with no
    relative insertion on $D_{\infty}$. The purpose of this
    subsection is to show that the corresponding relative invariant
    of $(P, D_{\infty})$ is zero.
    Namely, we would like to calculate certain admissible invariants with empty relative
    insertion,
    $$<\varpi, \tau_{i_1}([pt]), \tau_{i_2}(\beta_2[D_0]), \cdots,
    \tau_{i_k}(\beta_k[D_0])|\emptyset>^{P, D_{\infty}}_{A}, $$
with $m=\sum_t
        (i_t+1)\leq D_{0}\cdot A$.
    Since there are no relative insertions at $D_{\infty}$, we have $D_{\infty}\cdot A=0$. It implies
    that $A\in  H_2(D_0, \Z)$.
    The dimension condition is
    \begin{equation}
    \begin{array}{ll}
    &2(C_1(A)+D_0\cdot A+2|\nu|+n-2+k+l)\cr
    =&\sum_t
    i_t+2k+2n+\sum_t \deg(\beta_t)+\deg(\varpi).
    \end{array}
    \end{equation}
    Using assumption $m\leq D\cdot A$, we have
    \begin{equation}
    \begin{array}{ll}
    &2(C_1(A)+n-2+k+l)\cr
    \leq& 2n+\sum_t \deg(\beta_t)+\deg(\varpi).
    \end{array}
    \end{equation}

    \begin{theorem}
        Suppose that $0<\omega(\pi_*(A))\leq V$ or $\pi_*(A)$ is
        not a sum of a uniruled class and nonzero effective class of $D$ and $m=\sum_t
        (i_t+1)\leq D_0\cdot A$. Then,
$$<\varpi, \tau_{i_1}([pt]), \tau_{i_2}(\beta_2[D_0]), \cdots,
    \tau_{i_k}(\beta_k[D_0])|\emptyset>^{P, D_{\infty}}_{A}=0$$
    \end{theorem}

     Notice that unlike Theorem \ref{MP}, $\pi_*(A)$ could be a uniruled class.

    \begin{proof}
     If a constant vertex
    is a  $D_0-$vertex, it represents constant stable maps into $D_0$.
    If a constant vertex is a rubber vertex, it represents a
    multiple of fiber classes in the rubber. Such a rubber stable
    map necessarily contains relative insertions which contradicts to our
    assumption. Therefore,
    rubber constant vertices do not exist in our situation.

There are three cases to consider.

    {\bf Case 1}: Suppose that our bipartite graph consists of no edges and only
    one $D_0-$vertex. Since there are no rubber constant vertices, there are no other constant
    vertices either. In this case, the fixed point set is simply
    $\overline{\M}^D_{0, k+l}(A, pt, S_2, \cdots, S_k)$.

It is
    clear that the fixed part of the complex (\ref{complex}) is
    $$D_{f,D_0}: \{v\in \Omega^0(f^*TD_0), v(x_i)\in TS_i, 1\leq i\leq k\}\rightarrow
    \Omega^{0,1}(f^*TD_0),$$
while the moving part is the complex (\ref{complex}).

Therefore, as the fixed point set, its virtual fundamental
    agrees with $$[\overline{\M}^D_{0, k+l}(A, pt, S_2, \cdots,
    S_k)]^{vir}$$
    Furthermore, $N_{vir}=\ker D_{f,L}-\hbox{coker} D_{f,L}$.

    By the localization formula, the contribution of this graph is
    $$
    \int_{[\overline{\M}^D_{0, k+1}(A, pt, S_2, \cdots, S_k)]^{vir}}
    \frac{\prod_t \bar{\psi}^{i_t}_t\varpi}{e(N_{vir})}.
    $$

     The degree of the virtual fundamental class is
    $$2(C_1(A)+n-3+k+l)-2n-\sum_{i=2} \deg(\beta_i).$$
    Both absolute and relative insertions have no equivariant
    parameter. By the dimension condition of the relative invariant, the total degree of insertions is
    straightly larger than the degree of the virtual class.
 Therefore,
    the contribution is zero.

    {\bf Case 2}: There is a only one rubber vertex  with homology
    class $B$, several $D_0-$vertices including some constant $D_0-$vertices. The constant $D_0-$vertices must be connected to the rubber vertex by edges.
    But there may be other edges not connected to any constant vertex.
    Suppose that the partition for the rubber vertex is $\mu$. It is easy to see that
    $$|\mu|=D_0\cdot A, \quad A=B+|\mu| F.$$
    To simplify the notation, we first assume that there
are only relative absolute insertions.
    Then all the marked points lies either on $M_0$ or on an edge.

     Suppose that the vertex under consideration contains
    only marked points $i_1, \cdots, i_l$ from relative absolute insertions. It has an additional marked
    point from the edge which is a degree $m$ multiple cover of the fiber. The moduli space is clearly
    $$\overline{\M}_{0, l+1}\times (S_{i_1}\cap \cdots\cap
    S_{i_l}).$$
    Suppose that the number of constant vertices is $l'_{\mu}$. We also need to consider the edge containing
    a marked point, say $x_i$. Since the image of $x_i$ has to be in $S_i$, it contributes a factor
    $$2n-\deg(\beta_i)$$ to the dimension of fixed point loci. We treat it as  part of $M_0$.
    Let $k'$ be the number of such edges. It is clear that
    $$l'_{\mu}+k'\leq l_{\mu}.$$
    The dimension of $M_0$ is
$$2n(l'_{\mu}+k')-2n-\sum_{t}
    \deg(\beta_{t})+2(k-k')-4l'_{\mu}.$$
 Let us compute the virtual dimension for the rubber vertex
    $\overline{\M}^{\sim}_{0,0}(B,\mu)$.

     Let us compute the virtual dimension
    of the rubber moduli space $\overline{\M}^{\sim}_{0,0}(B,\mu|)$. A
    rubber stable map can be thought as a $\mathbb C^*$-equivalence class of stable maps to $P$
    relative to both $D_0, D_{\infty}$. The relative condition
    determines a unique lift of $B\in H_2(D_{\infty}, \Z)$ to a
    class of $P$. A moment of thought tells us that the lift is
    $A$ itself. Then,
    \begin{equation}
    \begin{array}{ll}&\dim_{\C}[\overline{\M}^{\sim}_{0,0}(B,
    \mu|\nu)]^{vir}\cr
    =&C_1(A)+D_0\cdot A+n+1-3+l_{\mu}-|\mu|-1\cr
    =&C_1(A)+n-3-l_{\mu}.
    \end{array}
    \end{equation}

    Adding back $\varpi$ and using the fact that $m\leq D\cdot A=|\mu|$, the dimension of fixed loci is
    \begin{equation}
    \begin{array}{ll}&2(C_1(A)+n-3+l-l_{\mu})+2(k-k')-4l'_{\mu}-2n-\sum_t
    \deg(\beta_t)\cr
    =&2(C_1(A)+n-2+k+l)-2n-\sum_t
    \deg(\beta_t)-2(k'+l_{\mu}+2l'_{\mu}-1)\cr
    <&\deg(\varpi).
    \end{array}
    \end{equation}
    Hence, the contribution is zero.

{\bf General case}: The general case contains at least two
    vertices with non-fiber homology classes. Let us consider the
    marked point corresponding to the point class. There are two
    cases, either on a nonconstant vertex or on a constant vertex.

    By Lemma \ref{constant vertex} we can transport the point $D$-insertion on
    a constant vertex to one on a nonconstant vertex. Suppose that the homology class of
    this
    nonconstant vertex is $B$. Then, $\omega(\pi_*(B))<\omega(\pi_*(A))$. By the argument of Theorem \ref{MP},
    the contribution is zero.

\end{proof}

\section{The proof of main theorem}
     In this section, we establish our main theorem.
For that purpose, we need an additional result from localization.

\subsection{A nonvanishing Theorem}

    In previous section, we prove a vanishing theorem for the admissible invariant of $(P, D_{\infty})$
    with empty relative insertion on $D_0$. In this section, we consider the sup-admissible
    case with $m=\sum_t (i_t+1)>D_0\cdot A$. The idea is that the invariant of $(P, D_{\infty})$ is
    no longer zero and the invariant of $D_0$ will contribute in a nontrivial way.
    Consider relative invariant
    $$<\varpi, \tau_{i_1}([pt]), \tau_{i_2}(\beta_2[D_0]), \cdots,
    \tau_{i_k}(\beta_k[D_0])|\empty>^{P,
    D_{\infty}}_{A}.$$
    The dimension condition is
    $$2(C_1(A)+D_0\cdot A+n+1+k+l-3)=2\sum_t i_t+2n+
    \sum_t \deg(\beta_t)+2k+\deg(\varpi).$$

    Let $m=\sum_t (i_t+1)$ be the multiplicity. Recall that it is (i)admissible if $m=D_0\cdot A$;
    (ii) sup-admissible if $m> D_0\cdot A$; (iii) sub-admissible if
    $m<D_0\cdot A$. As mentioned in the introduction, we can always
    define an equivalent admissible invariant from a sub-admissible
    invariant. Hence, it is enough to consider admissible and
    sup-admissible invariants only. The dimension condition of
    the divisor invariant
    $$<\iota^*\varpi, \tau_{i_1}([pt]), \tau_{i_2}(\beta_2), \cdots,
    \tau_{i_k}(\beta_k)>^{
    D_{\infty}}_{A}$$ is
    $$2(C_1(A)+n+k+l-3)=2\sum_t i_t+2n+
    \sum_t \deg(\beta_t)+\deg(\varpi).$$
    Both invariants are well-defined only when $k=D_0\cdot A+1$.

    However, there is a well-defined twisted Gromov-Witten
    invariants for all the cases.
     Our expectation is that in the sup-admissible case,
    the invariant of $(P, D_{\infty})$ is dominated by the twisted Gromov-Witten.
    We do not know yet how to
    carry out a correspondence with the twisted
    Gromov-Witten invariants. We hope to come back to it in the
    future.

    Instead, we use the above idea to study a specific sup-admissible case.

    \begin{theorem} \label{nonvanishing}
    Suppose that $0<\omega(A)\leq M$ and $k=D_0\cdot A+1$. Then,
    \begin{equation}
    \begin{array}
   {ll}
   & <\varpi, [pt][D_0], \beta_2 [D_0], \cdots, \beta_k [D_0]
    |\emptyset>^{P, D_{\infty}}_{A}
        \cr
        =&c<\iota^*\varpi, [pt], \beta_2, \cdots, \beta_k>^{D}_{A}
        \end{array}
        \end{equation}
    for $c\neq 0$.
    \end{theorem}

\begin{proof}

    The proof is similar to that of the second vanishing theorem.
     Applying localization
    formula, relative invariant is the non-equivariant limit of
    $$\sum_F\int_{[\M_F]^{vir}} \frac{
    \varpi}{e(N_{vir})}.$$

We divide into 3 cases.

    {\bf Case 1}: Suppose that our bi-partite graph consists of only
    one vertex corresponding to the ordinary stable map into $D_0$
    and no edges. In this case, the fixed point set is simply
    $\overline{\M}^D_{0, k+l}(A, pt, S_2, \cdots, S_k)$.
    By the localization formula, the contribution of this graph is
    $$
    \int_{[\overline{\M}^D_{0, k+1}(A, pt, S_2, \cdots, S_k)]^{vir}}
    \frac{\varpi}{e(N_{vir})}.
    $$
    Since $\varpi$ have no equivariant parameter,
    we only need to take non-equivariant limit of $\frac{1}{e(N_{vir})}$.
    The latter is precisely the Euler class $e_L$ appearing in
    twisted Gromov-Witten invariants. Therefore, the contribution
    is just the twisted Gromov-Witten invariant
    $$<\iota^*\varpi, [pt], \beta_2, \cdots,
    \beta_k>^{D,L}_{A}.$$
    In the case of $k=D_0\cdot A+1$,
$$\dim [\overline{\M}^D_{0, k+l}(A, pt, S_2, \cdots,
    S_k)]^{vir}=2\sum_t i_t+ \deg (\varpi).$$  We
    immediately obtain the contribution as
    $$c<\iota^*\varpi, [pt], \beta_2,
    \cdots, \beta_k>^D_A,$$ where $c$ is the constant term of
    $\frac{1}{e(N_{vir})}$. Note that $N_{vir}$ has rank zero and hence $\frac{1}{e(N_{vir})}$ is a degree zero
    equivariant class. Therefore, it has an expression
    $$
    c+t^{-1}\gamma_1+t^{-2}\gamma_4+\cdots+t^{-m}\gamma_{2m}+\cdots,
    $$
    where $\gamma_{2m}$ has degree $2m$. On the other hand, the equivariant Euler class
    $e(N_{vir})$ is an invertible element after inverting
    equivariant parameter $t$. It implies that $c\neq 0$.

       The precise value of $c$ can also be worked out by the method of \cite{FP}.
       However, we don't need it in our paper.

    {\bf Case 2}; There is only one  rubber vertex with homology
    class $B$, several $D_0-$vertices including some constant vertices.
    Using the computation in the same c of the second vanishing
    theorem, the dimension condition and the assumption $k=D_0\cdot
    A+1$,
        the virtual dimension of fixed loci is seen to be
    \begin{equation}
    \begin{array}{ll}
    &2(C_1(A)+n-3+k+l)-2n-\sum_{t} deg (\beta_t)-2(k'+l_{\mu}+l'_{\mu})\cr
    =&\deg(\varpi)-2(k'+l_{\mu}+2l'_{\mu})\cr
    <&\deg(\varpi).
    \end{array}
    \end{equation}
    The contribution in this case
    is therefore equal to zero.

    {\bf General case}: The general case contains at least two
    vertices with non-fiber homology classes. The proof is identical to that of the second
    vanishing theorem. We omit it.
\end{proof}

    Using the above non-vanishing theorem, we are ready to prove
    our main theorem.

\subsection{Proof}

We are finally able to give the proof of Theorem \ref{main1}, which
we restate here for the convenience of readers.

     \begin{theorem} Suppose $D$ is uniruled and $A$ is a minimal uniruled class of $D$
such that
\begin{equation}\label{k}
<\iota^*\alpha_1,\cdots, \iota^*\alpha_l, [pt], \beta_2, \cdots,
\beta_k>^D_{A}\neq 0
\end{equation} for $k\leq D\cdot A+1$, $\beta_i\in H^*(D;\mathbb R)$, and $\alpha_j\in H^*(X;\mathbb R)$. Then $(X,\omega)$ is uniruled.
\end{theorem}

\begin{proof}
Notice that all the insertions of the invariant (\ref{k}) are
non-descendent.
 By linearity  we can assume that each $\beta_i\in
\Theta$ and each $\alpha_j\in \Xi$. Recall that $\Theta$ is the
chosen self dual basis of $\oplus_{q=0}^{2n-2}H^q(D;\mathbb R)$, and
$\Xi$ is the chosen basis of $\oplus_{p=0}^{2n}H^p(X;\mathbb R)$.

  Assume first that  the number $k$ of insertions of (\ref{k})
  satisfies
$$k=D\cdot A+1.$$
Then the following relative invariant of $(P, D_{\infty})$,
\begin{equation}\label{X}
\langle \alpha_1,\cdots, \alpha_l, [pt], \tilde \beta_2,\cdots,
\tilde \beta_k \rangle_A^{P, D_{\infty}},
\end{equation}
 is well defined, as the dimension difference of the moduli spaces
$$2D\cdot A+2$$
matches with the difference of the total cohomology degree $2k$.
Moreover, by the minimality of the class $A$ and Theorem
\ref{nonvanishing}, the invariant (\ref{X}) is nonzero.

Notice that the invariant (\ref{X}) is the relative invariant
associated to the following standard sup-admissible graph
$\Gamma(\varpi|\mu)$ of the symplectic cut with
\begin{equation}
\varpi=(\alpha_1,\cdots, \alpha_l), \quad \mu=((1,[pt]),(1,\beta_2),
\cdots, (1, \beta_k)).
\end{equation}
 Hence the vector $v^{rel}_{D-pt}$ is nonzero.

By Theorem \ref{minimal}, the vector $v^{abs}_{D-pt}$ is nonzero as
well. Notice that all the relevant invariants of $X$ have a point
insertion (possibly descendent). Thus $X$ is uniruled by Theorem
\ref{descendent}.

For the general case that $k\leq D\cdot A+1$,  notice that we can
always increase the number of $D-$insertions by adding divisor
insertions of the form $PD(\omega_D)$ to achieve the equality.
\end{proof}

\begin{remark}
We think it is possible to weaken the assumption  by the existence
of a minimal descendent invariant of the form
$$<\iota^*\varpi, \tau_{i_1}([pt]), \tau_{i_2}(\beta_2), \dots,
    \tau_{i_k}(\beta_k)|\mu>^{D}_A\ne 0,$$
with $\sum_t(i_t+1)\leq D\cdot A+1$. What is missing is an analogue
of Theorem \ref{nonvanishing} in this set up.

We are able to weaken the assumption in a different direction in the
following subsection.
\end{remark}

\subsection{Weakening the minimal condition}

 The minimality condition on the uniruled class $A$ is a symplectic condition and is in
 agreement with
 the main theme of the paper conceptually.  However, given a tamed almost complex structure,
  it is possible to weaken somewhat the minimal condition.

If $J$ is a tamed almost complex structure,   a uniruled class $A$
is called  {$J-$indecomposable or just indecomposable} if it is not
the sum of another uniruled class and a nonzero $J-$effective class.
For example a uniruled class which is primitive and lies on an
extremal ray of the $J-$effective  cone is an indecomposable
uniruled class. A minimal uniruled class is obviously
indecomposable. Unfortunately, our argument does not extend to the
indecomposable situation. Instead, it applies to the following
intermediate  situation.

We call a uniruled class $A$ of $D$ {\em a globally indecomposable
uniruled class} of $(X, D)$ if $A$ is not a sum of a uniruled class
of $D$ and a nonzero effective class of $X$. Certainly a globally
indecomposable class uniruled class of $(X, D)$ is an indecomposable
uniruled class of $D$. The converse is not true, as a class of $D$
is not necessarily effective in $D$ even if it is effective in $X$.

Notice that a minimal uniruled class is globally  indecomposable.
Thus the following is a slightly stronger version of Theorem
\ref{main}.

\begin{theorem} \label{main} Suppose $A$ is a globally indecomposable uniruled class of $D$
such that
$$<\iota^*\alpha_1, \cdots,
\iota^*\alpha_l, [pt], \beta_2, \cdots, \beta_k >^D_{A}\neq 0$$ for
$k\leq D\cdot A+1$.  Then $X$ is uniruled.
\end{theorem}

     \begin{proof}
        We shall sketch the necessary modification of the proof while
        leaving the details for the interested readers.

        First step is to modify the partial order where we define
        $B\leq A$ if $A$ is a sum of $B$ and an effective class.

        Next, we modify the correspondence.
\begin{definition}\label{D-pt} We consider the following subset $I_{D-pt,0}$ of
colored standard graphs of $X$,

1. $g(\Gamma)=0$,

2. the fixed class $A$,

3. admissible graphs with a $D-$point insertion,

4. sup-admissible graphs of the form $\Gamma_0$
\begin{equation}\label{condition}
 \sum_{t=1}^k(i_t+1)=D\cdot A+1,
\end{equation}

5. empty graph excluded.

\end{definition}

   To prove the correspondence, we apply the degeneration formula as
   before. The point of attention is the nonzero term involving
   nontrivial graphs for both $\Gamma^+, \Gamma^-$. In this case,
   we express $A=A_1+A_2$, where $A_2=\pi_*(A(\Gamma^+))$ is the effective class of $D$,
   and $A_1$ is an effective class of $X$.
   By Theorem 4.1, $A(\Gamma^-)$ is either a fiber class or $A_2$ is
   a sum of a uniruled class and an effective class of $D$.  If $A_2$ is a sum of a uniruled class and an
   effective class of $D$, $A$ itself is a sum of a uniruled class of
   $D$ and an effective class of $X$. This contradicts the
   assumption. Therefore, only the case of $A(\Gamma^-)$ being a fiber
   class can appear.
But then $A(\Gamma^+)=A$, which is exactly
   what we want.

   The other situation is of course the special
   graph. The rest of proof goes through without change.

   \end{proof}

\section{Applications}
    As  mentioned in the introduction, our main theorem can be
    applied as an existence theorem of uniruled manifolds.  In this section, we apply our
     main theorem to construct uniruled symplectic manifolds
     inductively,  generalizing
    several early results of McDuff.

\subsection{4--dimensional uniruled divisors} \label{4dim}

A deep result in dimension 4 is that being uniruled is a smooth
property. More precisely, a $4-$manifold $(M, \omega)$ is uniruled
if and only if $M$ is diffeomorphic to a connected sum of $\mathbb
P^2$ or a $S^2-$bundle over a surface with a number of
$\overline{\mathbb P}^2$. Moreover, the isotopy class of $\omega$ is
determined by $[\omega]$.

We need to analyze  uniruled classes and the corresponding
insertions.

\begin{prop} \label{criterion} If $A$ is a uniruled class of a $4-$manifold, then

(i) $A$ is represented by an embedded symplectic surface,

(ii) $C_1(A)\geq 2$,

(iii) $A^2\geq 0$,

(iv) $A\cdot B\geq 0$ for any class $B$ with a non-trivial GW
invariant.

\end{prop}

\begin{proof}
The point is  that $4-$manifolds are semi-positive. Thus, for a
generic tamed almost complex structure $J$ and a generic point, only
{\it somewhere injective} $J-$holomorphic curves with a smooth
domain $S^2$ contribute to the relevant GW invariant (see
\cite{McS}). Such a curve can be smoothed to an embedded symplectic
surface.

The genus 0 moduli space of a class $A$ has dimension $C_1(A)-1$.
Since there is at least a point insertion, we have (ii).

Now the adjunction inequality in \cite{Mc4}, together with (i) and
(ii),  implies that $A\cdot A\geq 0$.

(iv) follows from (i), (iii) and positivity of intersection in
\cite{Mc4}. By (i) we have an embedded  $J-$holomorphic sphere  $C$
in the class $A$ for some tamed $J$. The class $B$ is represented by
a union $\cup m_iD_i$ of  possibly singular irreducible
$J-$holomorphic curves with multiplicities. If an irreducible
component $D_i$ is distinct from $C$, then apply the positivity of
intersection, if $D_i=C$ then apply (iii).

\end{proof}

\subsubsection{Simple case--proportional}
 For  $\mathbb P^2$, let $H$ be the generator of $H_2$ with
 positive area. $H$ is a uniruled class and any uniruled class of
 of the form $aH$ with $a>0$.
 Obviously,  $H$ is the minimal uniruled class. The
relevant insertion is $(pt, pt)$. As $pt$ is a restriction class,
i.e. an $\alpha$ class, we can take $k=1$.

Similarly, for the blow-up of an $S^2-$bundle over
 a surface of positive genus, the fiber class is a uniruled class,
 and any uniruled class is a positive multiple of the fiber class.
The relevant insertion for the fiber class  is $pt$. Thus again we
can take $k=1$.

It is easier to apply Theorem \ref{main} in this case.

\begin{cor} Suppose $(X^6, \omega)$ contains a divisor $D$
 which is diffeomorphic to $\mathbb P^2$ or the blow-up of a $S^2-$bundle over
 a surface of positive genus. If the normal bundle $N_{D}$ is
 non-negative on
 a unruled class, then $X$ is uniruled.
\end{cor}

\subsubsection{General case} For other $M^4$, the uniruled classes are
not proportional to each other. Thus the minimality condition
depends on the class of the symplectic form on  $M$.

 We
first analyze  the easier case of an $S^2-$bundles over $S^2$. For
$S^2\times S^2$, by uniqueness of symplectic structures,  any
symplectic form is of product form. Let $A_1$ and $A_2$ be the
classes of the factors with positive area. It is easy to see that
any uniruled class is of the form $a_1A_1+b_1A_2$ with $a_1\geq 0,
a_2\geq 0$. Thus either $A_1$ or $A_2$ has the minimal area.

For the nontrivial bundle $S^2\tilde \times S^2$, let $F_0$ be the
class of a fiber with positive area and $E$ be the unique $-1$
section class with positive
 area. If $aF_0+bE$ is a uniruled class then $b\geq 0$ by (iv) of
Proposition \ref{criterion}, since $F_0\cdot E=1, F_0\cdot F_0=0$.
And if $b> 0$, then, by (ii) of Proposition \ref{criterion}, $a\geq
1$ as $C_1(E)=1$. Thus $F_0$ is always the minimal uniruled class no
matter what the symplectic structure is.

Since the relevant insertion  for $A_1, A_2$ and $F_0$ is  just
$pt$, we have

\begin{cor}
Suppose $D=S^2\times S^2$ and the  restriction of the normal bundle
$N_{D}$  to the factor with the least area  is non-negative, then
$X$ is uniruled.

In the case of the non-trivial bundle,
 $X$ is uniruled if the
restriction of the normal bundle  $N_{D}$ to $F_0$ is non-negative.
\end{cor}

\begin{remark} The following restatement  is related to the contraction
criterion in \cite{tLR}. Suppose $(X^6, \omega)$ is  a non-uniruled
$6-$manifold containing an $S^2-$bundle $D$. Then, if $D= S^2\tilde
\times S^2$, the normal bundle is negative along the $S^2-$fibers,
and in the case of $S^2\times S^2$ where there are two
$S^2-$directions, the normal bundle is negative along the one with
the least area.
\end{remark}

The remaining uniruled $4-$manifolds are  $(\mathbb P^2 \#
k\overline{\mathbb P}^2,\tau)$ with $k\geq 2$. Let $H, E_i$ be the
generators of $H^2$ of the summands with positive area. Then
$[\tau]$ is of the form $uH-\sum_{i=1}^k v_i E_i$ with $u, v_i>0$.

$H, E_i$ all have non-trivial GW invariants. Thus, by (iv) of
Proposition \ref{criterion},   a class $\xi=aH-\sum_{i=1}^k b_i E_i$
is uniruled only if
\begin{equation}\label{symplectic}
a\geq 0, \quad b_i\geq 0.
\end{equation}
 The condition that
$\xi$ has square $0$ is
\begin{equation}\label{square0}
a^2=\sum_i b_i^2.
\end{equation}
If (\ref{square0}) is satisfied, then the condition that  $\xi$ is
represented by an embedded sphere of genus 0 is  given by the
adjunction formula
\begin{equation}\label{genus}
3a=\sum_i b_i+2.
\end{equation}

\begin{definition} Any primitive class $\xi$
satisfying (\ref{symplectic}), (\ref{square0}), (\ref{genus}) is
called a fiber class.
\end{definition}

\begin{theorem} Suppose $(D, \tau)$ is a symplectic $\mathbb P^2\#k\overline{\mathbb P}^2$ with $k\geq 1$.
  Then a  fiber class  is a uniruled class. In addition,
any uniruled class is the sum of a positive multiple of a fiber
class and another class with positive symplectic area.

 Consequently,
if $(D, \tau)$ is a symplectic divisor of a symplectic 6 manifold
$(X, \omega)$ and the normal bundle $N_D$ is non-negative on a fiber
class with the minimal $\tau-$area, then $X$ is uniruled.
\end{theorem}

\begin{proof}
 By \cite{LiL}, a fiber  class is equivalent to the indecomposable
uniruled class $H-E_1$ via diffeomorphisms preserving the canonical
class.

By Proposition \ref{criterion}, a uniruled class is represented by
an embedded symplectic surface with non-negative self-intersection.
By \cite{LiL} such a class is equivalent to a reduced class via
diffeomorphisms preserving the canonical class. A class
$\xi=aH-\sum_{i=1}^k b_i E_i$ is called reduced if $$a\geq
b_1+b_2+b_3, \quad b_i\geq b_{i+1}\geq 0.$$ It is  also shown in
\cite{LiL} if the surface is actually a sphere, then the class is
equivalent to
$$2H,\quad  H-E_1, \quad (l+1)H-lE_1, \quad (l+1)H-lE_1-E_2, \quad
l\geq 2.$$

The effective curve cone is generated by $-1$ classes.

 Hence it suffices to show that  a reduced class
with non-negative self-intersection is the sum of a positive
multiple of a fiber class and another class with positive symplectic
area.

\end{proof}

To enumerate fiber classes, we notice that
$$H-E_1=(H-E_1-E_2)+E_2,$$
i.e. it is the sum of two $-1$ classes whose intersection number is
equal to 1. Thus a fiber class is the sum of two $-1$ classes.

When $k\leq 8$, there are only finitely many $-1$ classes. So it is
straightforward though tedious  to list all the fiber classes.

Any class of the form $H-E_i, 1\leq i\leq k$ is such a class, and
when $k\leq 3$, there are no other classes.

When $k=4$, there is a new class with the coefficient of $H$ being
2,
        $$2H-E_1-E_2-E_3-E_4.$$ By choosing any  4 distinct numbers between
        1 and $k\geq 4$, we get many such classes for higher $k$.
        When $k=4, 5$, there are no other new classes.

When $k=6$, there are 6 new classes with the coefficient of $H$
being 3,
$$3H-2E_1-E_2-E_3-E_4-E_5-E_6,$$
and its permutations in the $E_i$. For higher $k$, there are similar
classes.

When $k=7$, there are additional  classes,
$$(4|2,2,2, 1,1,1,1), \quad (5|2,2,2,2,2,2,1)
$$
and their permutations in the $E_i$.

When $k=8$, there are  additional classes,
$$(4|3,1,1,1,1,1,1,1), \quad (5|3,2,2,2,1,1,1,1),
$$
$$(6|3,3,2,2,2,2,1,1), \quad (7|4,3,2,2,2,2,2,2),$$
and their permutations in the $E_i$.

\subsection{Higher dimensional case}

We start with the proof of Corollary \ref{homologically injective}.
\begin{proof}
The homologically injective case, which is equivalent to being
cohomology surjective, is clear.

For the case of a projective uniruled divisor,    according to
Theorem \ref{projective}, there is a nonzero invariant $I_{p,q}$ for
a minimal uniruled class $A$. We now only need to observe that
$[\omega|_D]^p=[\omega^p|_D]$.

\end{proof}




As we already mentioned that a Fano manifold is projectively
uniruled.
     In particular,  a hypersurface of $\mathbb P^n$ (for $n\geq 4$) of degree at most
     $n$ is Fano and hence projectively uniruled.

     \begin{cor} Suppose $n\geq 4$ and $D$ is a Fano hypersurface symplectic divisor of $X$.
       If $N_D=\lambda [\omega _D]$ with $\lambda\geq 0$ then $X$ is uniruled.
       \end{cor}

       \begin{proof}
  Since $n\geq 4$, by the Lefschetz
     hyperplane theorem, $D$ has $b_2=1$.
 As  $N_D=\lambda [\omega _D]$ for some
 $\lambda$  for any uniruled class  of
 $D$, and in particular a minimal uniruled class, the statement
 follows from Corollary \ref{homologically injective}.

       \end{proof}

Of course a particular case is $D=\mathbb P^{n-1}$ discussed in \S
2.

 In general case we still need to verify the minimal condition. Of
course the uniruled divisor needs not to be a projective manifold.
For instances, the divisor could be a rather general uniruled
fibration discussed in \S2. Let us treat the case of a symplectic
$\mathbb P^{k}-$bundle. Since the line class in the fiber is
uniruled, and the relevant insertions can be taken to be $(pt,
\omega|_D^k)$, we have

\begin{cor} Suppose $D$ is a symplectic divisor of $X$. If $D$ is a projective space bundle with
the fiber class being  the minimal uniruled class and  normal bundle
$N_D$  is non-negative
 along the fibers, then $X$ is unruled.
\end{cor}

McDuff also considered the case of product $\mathbb
P^k-$bundles in \cite{Mc2}.
A natural source of such $D$ is from blowing up a `non-negative'
$\mathbb P^k$ with a large trivial neighborhood. Suppose $\mathbb
P^k\subset X$
 has trivial normal bundle. Then the blow up along $\mathbb P^k$
has a divisor $D=\mathbb P^k\times \mathbb P^{n-k-1}$. The normal
bundle of $D$ along a line in $\mathbb P^k$ is trivial. Similar to
the case of $S^2\times S^2$, we can argue that
 the area of this line is minimal among all uniruled class
of $D$.
In particular, as observed by \cite{Mc2},  a symplectic $\mathbb P^1$ with
a sufficiently large product symplectic neighborhood can only exist in a uniruled
manifold.

In fact we can prove more.

\begin{cor} Suppose $S$ is a uniruled symplectic submanifold
whose minimal uniruled class has area $\eta$ and insertions all
being  restriction classes. If  $S$ has trivial symplectic
neighborhood of radius at least $\eta$. Then $X$ is uniruled.
\end{cor}

We will treat the more general case of `non-negative' normal bundle
in another paper on uniruled submanifolds.

\end{document}